\documentclass{siamart}

\usepackage{amsfonts}

\usepackage{amsmath,cite}
\usepackage{epsfig,epsf,latexsym,subfigure}

\numberwithin{equation}{section}
\numberwithin{table}{section}
\numberwithin{figure}{section}

\theoremstyle{plain}
\newtheorem{scheme}{Scheme}[section]
\newtheorem{thm}{Theorem}[section]

\newtheorem{rem}{Remark}[section]

\renewcommand\theequation{\thesection.\arabic{equation}}
\theoremstyle{plain}

\newcommand{\be}{\begin{equation}}
\newcommand{\ee}{\end{equation}}

\newcommand{\bse}{\begin{subequations}}
\newcommand{\ese}{\end{subequations}}

\def\cF{\mathcal{F}}

\def\bn{\mathbf{n}}

\def\be{\mathbf{e}}
\def\bb{\mathbf{b}}
\def\bx{\mathbf{x}}
\def\ba{\mathbf{a}}

\def\bD{\mathbf{D}}

\def\bI{\mathbf{I}}

\def\bD{\mathbf{D}}

\def\cL{\mathcal{L}}

\def\cG{\mathcal{G}}
\def\cA{\mathcal{A}}
\def\cB{\mathcal{B}}
\def\cN{\mathcal{N}}

\newcommand{\ben}{\begin{eqnarray}}
\newcommand{\een}{\end{eqnarray}}
\newcommand{\beq}{\begin{equation}}
\newcommand{\eeq}{\end{equation}}
\newcommand{\bea}{\begin{array}}
\newcommand{\eea}{\end{array}}
\newcommand{\bef}{\begin{figure}}
\newcommand{\eef}{\end{figure}}

\title{ Arbitrarily High-order Unconditionally Energy Stable Schemes for Thermodynamically Consistent Gradient Flow Models}
\author{ Yuezheng Gong \thanks{College of Science, Nanjing University of Aeronautics and Astronautics, Nanjing 210016, China; Email: gongyuezheng@nuaa.edu.cn.}
\and Jia Zhao \thanks{Department of Mathematics \& Statistics, Utah State University, Logan, UT 84322, USA; Email: jia.zhao@usu.edu.}
\and Qi Wang \thanks{ Beijing Computational Science Research Center, Beijing, China, 100193; Department of Mathematics, University of South Carolina, Columbia, SC 29208, USA; School of Materials and Engineering, Nankai University, Tianjin, China, 300084; Email: qwang@math.sc.edu.}}

\begin{document}
\maketitle

\begin{abstract}
We present a systematical approach to developing arbitrarily high order, unconditionally energy stable numerical schemes for thermodynamically consistent gradient flow models that satisfy energy dissipation laws.  Utilizing the energy quadratization (EQ) method, We formulate the gradient flow model into an equivalent form with a corresponding quadratic free energy functional. Based on the equivalent form with a quadratic energy, we propose two classes of energy stable numerical approximations. In the first approach, we use a prediction-correction strategy to improve the accuracy of linear numerical schemes. In the second approach, we adopt the Gaussian collocation method to discretize the equivalent form with a quadratic energy, arriving at an arbitrarily high-order scheme for gradient flow models. Schemes derived using both approaches are proved rigorously to be unconditionally energy stable. The proposed schemes are then implemented in four gradient flow models numerically to demonstrate their accuracy and effectiveness. Detailed numerical comparisons among these schemes are carried out as well. These numerical strategies are rather general so that they can be readily generalized to solve any thermodynamically consistent PDE models.
\end{abstract}

\maketitle

\section{Introduction}
In nonequilibrium thermodynamics, transient dynamics is stipulated primarily following the "linear response theory", collectively known as the Onsager principle, or equivalently the second law of thermodynamics  \cite{Onsager1931, Li&WangJAM2014, Yang&Li&Forest&Wang2016}.
The second law of thermodynamics has been used to develop non-equilibrium models for a wide range of material systems, ranging from life science, materials science, industrial processing, and engineering. The gradient flow model describing relaxation dynamics is one of the well-known examples, and other examples include most differential constitutive laws for complex fluid flows \cite{Bird87}. The models derived from the second law of thermodynamics are referred to as the thermodynamically consistent models.

For a given thermodynamically consistent, gradient flow model in the form of partial differential equations (PDEs), a computationally efficient, energy-dissipation-rate preserving, and high order numerical approximation is always desired. In the literature, the numerical scheme that preserves the energy dissipation property is known as the energy stable scheme \cite{Eyre1998}. If the energy-dissipation-rate preserving property doesn't depend on the time-step size, the schemes are called unconditionally energy stable.

In the past,  two widely used, distinct strategies for developing (unconditionally) energy stable schemes have been proposed, which are the convex splitting method \cite{Eyre1998,Wang&WiseSINUMA2011,Shen&Wang&Wise2012,Wang&Wang&WiseDCDS2010,Han&WangJCP2014,WiseSINUMA2009,Chen&Wang&Wang&WiseJSC2014} and the stabilizing method \cite{Zhao&Li&Wang&YangJSC2017,Li2016Characterizing,Shen10_1,Shen2010,ShenJ4,Zhao&Wang&YangCMAME2016,Feng&Tang&YangEAJAM2013}. The convex-splitting strategy relies on the existence of a pair of convex components that give rise to the free energy functional as a difference of the two convex functionals. If such a splitting exists in the free energy functional, a nonlinear scheme can be devised to render an unconditionally energy stable scheme.  The convex splitting schemes have been widely applied to the gradient flow models \cite{WangChen-CH, Eyre1998, Wang&WiseSINUMA2011, Shen&Wang&Wise2012, Wang&Wang&WiseDCDS2010, Han&WangJCP2014, WangCheng-1,WangCheng-2, WangCheng-3,WangCheng-4,WangCheng-6}. The stabilizing approach augments discretized equations by high order terms to turn the scheme into an energy stable scheme. Usually, this is accomplished by adding additional dissipation to the numerical scheme. Both strategies can yield a dissipative scheme, but do not guarantee to preserve the energy dissipation rate. Some other related work include \cite{Han&WangJCP2014,Gomez&HughesJCP2011,Gomez&NogueiraCMAME2012,Feng&Wang&Wise&Zhang2017,Guo&Lin&Lowengrub&Wise2017,Ju&Li&Qiao&ZhangMC2017,Guillen&Tierra2014,Gomez-crystal} .

Recently, Badia, Guillen-Gonzales, Gutierres-Santacreu, and Tierra pioneered a new idea of transforming the free energy into a quadratic functional to derive energy stable schemes \cite{Badia&Guillen-Gonzalez&Gutierrez-SantacreuJCP2011, Guillen&Tierra2014}. This is amplified and systematically applied to many specific thermodynamic and hydrodynamic models by Yang, Zhao, Shen and Wang, etc. \cite{Yang-EQ-PFC, Yang&ZhaoViscousCH, Yang&Zhao&Wang&ShenM3AS2017, Zhao&Wang&YangCMAME2016, Zhao&Yang&Gong&WangCMAME2017, Gong&Zhao&WangACM, Gong&Zhao&Yang&WangSISC}.
Yang, Zhao, and Wang coined the name Invariant Energy Quadratization (IEQ) method for this class of methods. Later, we abbreviated the name to simply Energy Quadratization (EQ) method, which is more appropriate. This strategy bypasses the complications in the other methods to derive linear, second-order energy stable schemes in time readily. It is so general that the EQ approach has little restriction on the specific expression of the free energy functional. Recently, Shen et al. \cite{SAV-1, SAV-2} implemented the idea of EQ using a scalar auxiliary variable and called it the SAV method. However, using either EQ or SAV strategy, one has only designed and proved rigorously unconditionally energy stable numerical approximations up to second-order in time so far. This may not be sufficient for some gradient flow problems with long time simulations or sharp transition dynamics. Recently in \cite{RKJCP2017}, the authors propose a high order Runge-Kutta (RK) method for gradient flow problems based on convex splitting methods, but it requires too many stages even to reach 3rd order accuracy, and it does not work for gradient flows with variable mobilities.

In this paper, we develop systematically two classes of numerical approximations exploiting the EQ approach for thermodynamically consistent gradient flow models. In the first strategy, we introduce a prediction step to correct the extrapolation values (in the explicit terms of the linear schemes resulted from the EQ approach).  Specifically, instead of using extrapolations to obtain explicit terms at the desired time level \cite{Yang&ZhaoViscousCH, Yang&Zhao&Wang&ShenM3AS2017, Zhao&Wang&YangCMAME2016, Zhao&Yang&Gong&WangCMAME2017, Gong&Zhao&WangACM, Gong&Zhao&Yang&WangSISC}, we use a fixed-point iteration to predict and correct the terms using values from previous iteration steps. In numerical experiments, we show that this strategy can improve the accuracy of numerical schemes significantly. In the second strategy, we extend the second-order EQ method \cite{Yang&ZhaoViscousCH, Yang&Zhao&Wang&ShenM3AS2017, Zhao&Wang&YangCMAME2016, Zhao&Yang&Gong&WangCMAME2017, Gong&Zhao&WangACM, Gong&Zhao&Yang&WangSISC}
into an arbitrarily high-order energy quadratization (HEQ) method. The HEQ strategy consists of two parts. Firstly, by introducing auxiliary variables, we transform the original gradient flow model into a gradient flow model with an energy functional consisting of only quadratic terms of the unknown variables. The resulting model is referred to as the quadratized gradient flow model. Secondly, we apply the Gaussian collocation method to the quadratized gradient flow model to produce unconditionally energy stable numerical schemes. It turns out that the classical Crank-Nicolson and backward differential formula schemes are special cases of the newly proposed general scheme. Our new numerical strategy provides an elegant solution for developing arbitrarily high order and unconditionally energy stable numerical schemes for gradient flow models. Moreover, the schemes preserve the energy dissipation rate in the transformed variables accurately. Due to their high-order accuracy and unconditional energy stability, these schemes can allow large time steps, making the numerical approximations especially appealing for long time computations.

We organize the rest of the paper as follows. In \S \ref{sec:Model}, we reformulate gradient flow models by using the EQ method. In \S \ref{sec:timeDis}, we present two classes of unconditionally energy stable, numerical approximations for the reformulated EQ model. In \S \ref{sec:spaceDis}, the Fourier pseudospectral method is employed to give rise to the spatial discretization. Then four numerical examples of gradient flow models are shown to validate the efficiency and accuracy of our proposed schemes in \S \ref{sec:Numer}. Finally, we give a concluding remark in the last section.

\section{Gradient Flow Models and Their EQ Reformulation}\label{sec:Model}

We present the general thermodynamically consistent, gradient flow model firstly. Then, we reformulate the general gradient flow system into an equivalent form with a quadratic energy functional using the energy quadratization technique, called EQ reformulation. The EQ reformulation for this class of gradient flow models provides an elegant platform for developing arbitrarily high-order unconditionally energy stable schemes. In this paper, we adopt periodic boundary conditions for simplicity. The results can be readily applied to gradient flow problems with physical boundary conditions so long as the spatial discretization respects the integration-by-parts or summation-by-parts formula.

\subsection{Gradient flow models}

Mathematically, the general form of the governing system of equations of a gradient flow model is given by \cite{Zhao2018EQreview,SAV-2}
\beq \label{eq:gradient-flow}
\frac{\partial}{\partial t}\Phi = \cG \frac{\delta F}{\delta \Phi},
\eeq
where $\Phi=(\phi_1,\cdots,\phi_s)^T$ are the state variables, $\cG$ is an $s$-by-$s$ mobility matrix operator which is negative semi-definite and may depend on $\Phi$. Here $F$ is the effective free energy of the material system, and $\frac{\delta F}{\delta \Phi}$ is the variational derivative of the free energy functional with respect to the state variables, known as the chemical potential. Then, the triple $(\Phi,\cG,F)$ uniquely defines a gradient flow model. One intrinsic property  of \eqref{eq:gradient-flow} is the energy dissipation law
\beq\label{EDL}
\frac{d F}{d t} = \left( \frac{\delta F}{\delta \Phi}, \frac{\partial\Phi}{\partial t} \right) = \left( \frac{\delta F}{\delta \Phi}, \cG \frac{\delta F}{\delta \Phi} \right) \leq 0,
\eeq
where the inner product is defined by $({\bf f}, {\bf g}) = \sum\limits_{i=1}^s \int_\Omega f_i g_i d\bx$, $\forall \mathbf{f}, \mathbf{g} \in (L^2(\Omega))^s$, and $\Omega$ is the material domain. Note that the energy dissipation law \eqref{EDL} holds only for suitable boundary conditions. These boundary conditions include the periodic boundary conditions and the boundary conditions that make the boundary integrals resulted during the integration by parts vanish.

\subsection{Model reformulation using the EQ approach}

We reformulate the general gradient flow model \eqref{eq:gradient-flow} by first transforming the free energy into a quadratic form. We illustrate  the idea using a simple case, where the  free energy is given by
\beq\label{initial-energy}
F = \frac{1}{2}(\Phi, \cL \Phi ) + \big( f (\Phi) ,1\big).
\eeq
Here $\cL$ is a linear, self-adjoint, positive definite operator and $f$ is the bulk part of the free energy density, which is bounded from below for physically accessible state of $\Phi$. Then the  free energy $F$ can be rewritten into
\beq\label{EQ-energy}
\cF = \frac{1}{2}(\Phi, \cL \Phi) + \frac{1}{2}\|q\|^2 - A,
\eeq
where  $q = \sqrt{2\Big(f (\Phi) +\frac{A}{|\Omega|}\Big)}$, and $A$ is a constant large enough to make $q$ well-defined.  Here $\| \bullet \|$ represents the $L^2$ norm, i.e. $\| f\| = \sqrt{ \int_\Omega f^2 d\bx }$, $\forall f \in L^2(\Omega)$.

For instance, given a Ginzburg-Landau free energy
\beq\label{Ginzburg-Landau-energy}
F = \frac{\varepsilon^2}{2}\|\nabla \phi\|^2 + \frac{1}{4}\|\phi^2-1\|^2,
\eeq
if we conduct an integration by part once, it converts to
\beq
F = \frac{1}{2} \Big( -\varepsilon^2 \Delta \phi+\gamma_0 \phi, \phi  \Big) + \frac{1}{4}\|\phi^2-1-\gamma_0\|^2 + \frac{\varepsilon^2}{2}\int_{\partial \Omega} \phi \nabla \phi \cdot \bn d S - \left(\frac{\gamma_0}{2}+ \frac{\gamma_0^2}{4}\right)|\Omega|,
\eeq
we identify $\cL=-\varepsilon^2 \Delta+\gamma_0$ and $q = \frac{1}{\sqrt{2}}( \phi^2-1-\gamma_0)$,  assuming the boundary integral term vanishes.

Denote $g(\Phi) = \sqrt{2\left(f (\Phi) +\frac{A}{|\Omega|}\right)}$. Then we reformulate model \eqref{eq:gradient-flow} to an equivalent form
\beq \label{eq:gradient-flow-EQ-scalar}
\left\{
\bea{l}
\frac{\partial}{\partial t}\Phi = \cG \Big(\cL \Phi + q \frac{\partial g}{\partial \Phi} \Big), \\
\frac{\partial}{\partial t} q = \frac{\partial g}{\partial \Phi} \cdot \frac{\partial \Phi}{\partial t},
\eea
\right.
\eeq
where $\ba \cdot \bb = \sum\limits_{i=1}^s a_i b_i.$ Letting $\Psi = \left( \bea{l} \Phi \\ q \eea \right)$, system \eqref{eq:gradient-flow-EQ-scalar} can be written in the following compact form
\beq \label{eq:gradient-flow-EQ-full}
\frac{\partial}{\partial t} \Psi = \cN(\Psi) \cB \Psi,
\eeq
where $\cN(\Psi) = \cA^* \cG \cA$ is an $(s+1)\times(s+1)$ matrix operator depending on  $\Psi$, $\cA^*$ is the adjoint operator of $\cA,$ and $$\cA = \left(\bI_s \quad \frac{\partial g}{\partial \Phi}\right)_{s\times(s+1)},\quad \cB=\textrm{diag}(\cL, 1)_{(s+1)\times(s+1)}.$$ In this case, we have $\cA^* = \cA^T.$ Since $\cG$ is negative semi-definite, it can be  shown easily that $\cN(\Psi)$ is also negative semi-definite for any $\Psi$. In addition, $\cB$ is a linear, self-adjoint, positive definite operator thanks to the property of $\cL$. Define the $\cB$-norm as
\beq \label{eq:B-norm}
\|\Psi\|_{\cB} = \sqrt{( \Psi, \cB \Psi )},
\eeq
then the free energy \eqref{EQ-energy} is rewritten as $\cF = \frac{1}{2}\|\Psi\|_{\cB}^2-A.$ System \eqref{eq:gradient-flow-EQ-full} preserves the following energy dissipation law
\beq
\frac{d \cF}{d t} = \left( \frac{\delta \cF}{\delta \Psi} , \frac{\partial\Psi}{\partial t} \right) = \left( \cB \Psi , \cN(\Psi)  \cB\Psi \right) \leq 0.
\eeq

The original gradient flow system defined by $(\Phi,\cG, F)$ is transformed into a new system $(\Psi,\cN,\cF),$ where the free energy $\cF$ in the new system is regarded as a quadratic functional of $\Psi$. Therefore, this process is called energy quadratization reformulation.

\begin{rem}
More generally, if $q = g(\Phi,\nabla \Phi)=\sqrt{2\left(f(\Phi,\nabla \Phi) +\frac{A}{|\Omega|}\right)}$, then the equivalent EQ form is given by
\beq \label{eq:gradient-flow-EQ-scalar-2}
\left\{
\bea{l}
\frac{\partial}{\partial t}\Phi = \cG \Big(\cL \Phi + q \frac{\partial g}{\partial \Phi} - \nabla \cdot ( q \frac{\partial g}{\partial \nabla \Phi})\Big), \\
\frac{\partial}{\partial t} q = \frac{\partial g}{\partial \Phi}\cdot \frac{\partial \Phi}{\partial t} + \frac{\partial g}{\partial \nabla \Phi}\cdot \nabla \frac{\partial \Phi}{\partial t}.
\eea
\right.
\eeq
The system can again be written into the form \eqref{eq:gradient-flow-EQ-full} with the operators $\cA$ and $\cA^*$ given by
\beq
\cA = \left(\bI_s \quad \frac{\partial g}{\partial \Phi} - \nabla \cdot \frac{\partial g}{\partial \nabla \Phi}\right),\quad \cA^* = \left( \bea{l} \bI_s \\ \Big(\frac{\partial g}{\partial \Phi}+ \frac{\partial g}{\partial \nabla \Phi}\cdot \nabla\Big)^T \eea \right).\eeq
\end{rem}

\begin{rem}
In general, the EQ reformulation approach can be applied to any gradient flow models with free energies of high-order spatial derivatives so long as they are thermodynamically consistent and energy dissipative.
\end{rem}

We next discuss how to design accurate and energy stable schemes for gradient flow models. We will demonstrate that the equivalent form in \eqref{eq:gradient-flow-EQ-full} can be handled more easily than its original system given in \eqref{eq:gradient-flow}.

\section{Temporal Discretization}\label{sec:timeDis}

In this section, a class of linear second-order prediction-correction schemes and a class of arbitrarily high-order Gaussian collocation schemes are proposed respectively, where all the schemes are shown to be unconditionally energy stable, i.e., the energy dissipation property is conserved for any time step sizes at the semi-discrete level.

\subsection{Linear energy stable schemes}

As we know, the linear-implicit Crank-Nicolson (LCN) scheme and the linear-implicit second-order backward differentiation/extrapolation (LBDF2) method can be applied directly for discretizing the reformulated EQ system \eqref{eq:gradient-flow-EQ-full} in time to obtain linear unconditionally energy stable schemes \cite{Gong&Zhao&WangACM, Yang&Zhao&WangJCP2017, Guillen&Tierra2014, Yang&Zhao&Wang&ShenM3AS2017}. Even though both schemes are second-order accurate, the rate of convergence usually can be reached when the time step is small. Here, we propose a new class of prediction-correction schemes motivated by the works in \cite{predictor-corrector, ConvexSplittingPredictor, Liu&ShenMMAS2013}. Employing the prediction-correction strategy for the CN scheme or the BDF2 scheme, we obtain the following prediction-correction schemes:

\begin{scheme}[Linear Prediction-Correction Scheme] \label{scheme:PC}
Given $\Psi^{n-1}$ and $\Psi^n$, $\forall n\geq 1$, we obtain $\Psi^{n+1}$ through the following two steps:
\begin{enumerate}
\item  Prediction: predict $\Psi^{n+1}_{*}$ via some efficient and at least second order accurate numerical schemes.

\item  CN correction:
\beq\label{CN-PC}
\frac{\Psi^{n+1} - \Psi^n}{\Delta t} =
\cN\Big(\frac{\Psi^{n+1}_* + \Psi^n}{2}\Big) \cB \frac{\Psi^{n+1} + \Psi^n}{2};
\eeq
or BDF2 correction:
\beq\label{BDF2-PC}
\frac{3\Psi^{n+1} - 4\Psi^n +\Psi^{n-1}}{2 \Delta t} =
\cN(\Psi^{n+1}_*) \cB \Psi^{n+1}.
\eeq
\end{enumerate}
\end{scheme}
For each time step, several prediction strategies can be devised so long as the predicted value $\Psi^{n+1}_{*}$ is at least second-order consistent in time. As an illustration, we list some choices for the prediction step below.

\noindent $\bullet$ Case 1: if we set
\beq \label{CN-BDF-P0}
\Psi^{n+1}_{*} = 2\Psi^{n}-\Psi^{n-1},
\eeq
then  scheme \eqref{CN-PC} and \eqref{BDF2-PC} reduce to LCN and LBDF2 schemes studied in \cite{Gong&Zhao&WangACM}, respectively.

\noindent $\bullet$ Case 2: we set $\Psi^{n+1}_0 = 2\Psi^{n}-\Psi^{n-1}.$ For $i=0$ to $N-1$, we compute $\Psi^{n+1}_{i+1}$ using
\beq\label{CN-P1}
\frac{\Psi^{n+1}_{i+1} - \Psi^n}{\Delta t} =
\cN\Big(\frac{\Psi^{n+1}_i + \Psi^n}{2}\Big) \cB \frac{\Psi^{n+1}_{i+1} + \Psi^n}{2};
\eeq
or
\beq\label{BDF2-P1}
\frac{3\Psi^{n+1}_{i+1} - 4\Psi^n +\Psi^{n-1}}{2 \Delta t} =
\cN(\Psi^{n+1}_i) \cB \Psi^{n+1}_{i+1}.
\eeq
If  $\|\Psi^{n+1}_{i+1}-\Psi^{n+1}_{i}\|_\infty < \varepsilon_0$ and $i+1<N$, we stop the iteration and  set $\Psi^{n+1}_{*} = \Psi^{n+1}_{i+1}$; otherwise, we set $\Psi^{n+1}_{*} = \Psi^{n+1}_N.$

\noindent $\bullet$ Case 3: we rewrite $\cN(\Psi)$ into the sum of $\cN_1$, a linear  operator of constant coefficients  and a nonlinear operator $\cN_2(\Psi)$, $\cN(\Psi) = \cN_1 + \cN_2(\Psi)$.
Set $\Psi^{n+1}_0 = 2\Psi^{n}-\Psi^{n-1}.$ For $i=0$ to $N-1$, we compute $\Psi^{n+1}_{i+1}$ by
\beq\label{CN-P2}
\frac{\Psi^{n+1}_{i+1} - \Psi^n}{\Delta t} = \cN_1 \cB \frac{\Psi^{n+1}_{i+1} + \Psi^n}{2} +
\cN_2\Big(\frac{\Psi^{n+1}_i + \Psi^n}{2}\Big) \cB \frac{\Psi^{n+1}_i + \Psi^n}{2};
\eeq
or
\beq\label{BDF2-P2}
\frac{3\Psi^{n+1}_{i+1} - 4\Psi^n +\Psi^{n-1}}{2 \Delta t} = \cN_1 \cB \Psi^{n+1}_{i+1} + \cN_2(\Psi^{n+1}_i) \cB \Psi^{n+1}_i.
\eeq
If  $\|\Psi^{n+1}_{i+1}-\Psi^{n+1}_{i}\|_\infty < \varepsilon_0$ and  $i+1<N$, we stop the iteration and  set $\Psi^{n+1}_{*} = \Psi^{n+1}_{i+1}$; otherwise, we set $\Psi^{n+1}_{*} = \Psi^{n+1}_N.$

\begin{rem}
In some practical implementations, Case 3 is more efficient than Case 2 because $\cN_1$ is a linear operator of constant coefficients so that the Fast
Fourier transform (FFT) can be readily applied to the linear part of the scheme.
\end{rem}

\begin{rem}
If $N$ is large enough, the prediction-correction scheme \eqref{CN-PC} with \eqref{CN-P1} (or \eqref{CN-P2}) approximates to the fully implicit CN scheme while the prediction-correction scheme \eqref{BDF2-PC} with \eqref{BDF2-P1} (or \eqref{BDF2-P2}) approximates to the traditional BDF2 scheme. There is no theoretical result on the choice of iteration step $N$. From our numerical tests, several iteration steps $N \leq 5$ would improve the accuracy noticeably.
\end{rem}

\begin{rem}
In Scheme \ref{scheme:PC}, the initial second level datum $\Psi^1$ is computed by
\beq
\frac{\Psi^1 - \Psi^0}{\Delta t} =
\cN\Big(\frac{\Psi^1_* + \Psi^0}{2}\Big) \cB \frac{\Psi^1 + \Psi^0}{2},
\eeq
where $\Psi^1_*$ is given by
\beq
\frac{\Psi^1_* - \Psi^0}{\Delta t} = \cN_1 \cB \frac{\Psi^1_* + \Psi^0}{2} +
\cN_2(\Psi^0) \cB \Psi^0.
\eeq
\end{rem}

Similar to the linear unconditionally energy stable schemes in \cite{Gong&Zhao&WangACM}, we have the following theorems.
\begin{thm}
Scheme \eqref{CN-PC} satisfies the following energy dissipation law
\beq\label{CN-PC-EDL}
\frac{F^{n+1} - F^n}{\Delta t}  = \left( \cB \Psi^{n+\frac{1}{2}},   \cN\Big(\frac{\Psi^{n+1}_* + \Psi^n}{2}\Big) \cB \Psi^{n+\frac{1}{2}} \right)\leq 0,~ \forall n\geq 1,
\eeq
with $F^{n} = \frac{1}{2}\| \Psi^{n}\|_{\cB}^2-A$ and $\Psi^{n+\frac{1}{2}} = (\Psi^{n+1}+\Psi^{n})/2.$
\end{thm}


\begin{thm}
Scheme \eqref{BDF2-PC} satisfies the following energy identity
\beq\label{BDF2-PC-EDL}
\frac{F^{n+\frac{3}{2}} - F^{n+\frac{1}{2}} + \widetilde{F}^{n+1}}{\Delta t} =  \Big( \cB \Psi^{n+1},   \cN(\Psi_{*}^{n+1}) \cB \Psi^{n+1} \Big)\leq 0,~ \forall n\geq 1,
\eeq
with $F^{n+\frac{1}{2}} = \frac{1}{4}\Big( \|\Psi^{n}\|_{\cB}^2 + \|2\Psi^{n} - \Psi^{n-1}\|_{\cB}^2 \Big)- A $ and $\widetilde{F}^{n+1} = \frac{1}{4}\|\Psi^{n+1} - 2\Psi^{n} + \Psi^{n-1}\|_{\cB}^2.$
\end{thm}


\begin{rem}
Eqs. \eqref{CN-PC-EDL} and \eqref{BDF2-PC-EDL} imply that the two schemes given by \eqref{CN-PC} and \eqref{BDF2-PC} are unconditionally energy stable, i.e. they possess the energy decay property, respectively, $F^{n+1}\leq F^n,~ \forall n\geq 1$ and $F^{n+\frac{3}{2}}\leq F^{n+\frac{1}{2}},~ \forall n\geq 1.$ Therefore, they are unconditionally energy stable. In addition, Scheme \ref{scheme:PC} preserves the energy dissipation rate.
\end{rem}

\subsection{Arbitrarily high-order, unconditionally energy stable schemes}

The linear prediction-correction semi-discrete schemes discussed above are at most  second-order accurate in time. To derive arbitrarily high-order linear schemes, one can combine high-order backward differentiation formula with the matched extrapolation. Define the backward  difference operator for the $k$th order derivative as follows:
\beq
\Lambda_k \Psi^{n+1} = \sum_{i=0}^k \lambda_i^k \Psi^{n+1-i},
\eeq
where $\Lambda_0 \Psi^{n+1} :=\Psi^{n+1}$.  The values of $\{\lambda_i^k\}$ could be found in \cite{HairerBook}.
Then, we  propose the linear-implicit BDF-$k$  scheme (where  $k$ is the order of the scheme)
\begin{scheme}[BDF-$k$ Scheme]  \label{scheme:GradientFlow-BDFK}
Given $\Psi^{n+1-k}, \Psi^{n+2-k}, \cdots, \Psi^n$, we obtain $\Psi^{n+1}$ using
\beq
\frac{1}{\Delta t} \Lambda_k \Psi^{n+1}  = \cN(\overline{\Psi}^{n+1}) \cB \Psi^{n+1},
\eeq
where $\overline{(\bullet)}^{n+1}$ is a matched extrapolation with values from previous time steps.
\end{scheme}

Unfortunately, we are not able to prove energy stability for Scheme \ref{scheme:GradientFlow-BDFK} presently, although in practice, Scheme \ref{scheme:GradientFlow-BDFK} is usually shown to deliver an energy decay numerical result with reasonable time steps.

Given that we can't prove energy stability for the high order schemes obtained using the BDF method in time, we turn to another time discretization strategy for developing energy-stable schemes in time. Starting from the reformulated EQ system \eqref{eq:gradient-flow-EQ-full}, we apply the Gaussian collocation methods to construct arbitrarily high-order schemes in time. Then, we can prove the obtained schemes are unconditionally energy stable rigorously.

Recall the energy-quadratized system \eqref{eq:gradient-flow-EQ-full}, reformulated from the general gradient flow model \eqref{eq:gradient-flow}, as follows
\beq\label{eq:DE}
\frac{\partial }{\partial t}\Psi =\cN(\Psi) \cB \Psi.
\eeq
Firstly, we briefly recall the Runge-Kutta (RK) and collocation method (see Chapter \uppercase\expandafter{\romannumeral2} of \cite{HairerBook} for detailed discussions).
Applying an $s$-stage RK method to solve equation \eqref{eq:DE}, we obtain the following high-order energy quadratization (HEQ) scheme.

\begin{scheme}[$s$-stage HEQ-RK Method] \label{scheme:RK-method}
Let $b_i$, $a_{ij}$ ($i,j = 1,\cdots,s$) be real numbers and let $c_i = \sum\limits_{j=1}^s a_{ij}$.
Given $\Psi^n$, $\Psi^{n+1}$ is calculated by
\beq\label{HEQRK}
\bea{l}
k_i =  \cN \Big(\Psi^n+ \Delta t \sum\limits_{j=1}^s a_{ij} k_j \Big) \cB \Big( \Psi^n+ \Delta t \sum\limits_{j=1}^s a_{ij} k_j \Big), \quad i=1,\cdots,s, \\
\Psi^{n+1}  =  \Psi^n  + \Delta t \sum\limits_{i=1}^s b_i k_i.
\eea
\eeq
\end{scheme}
\noindent The coefficients are given by a Butcher table
\[
\begin{array}
{c|c}
\mathbf{c} & \mathbf{A} \\
\hline
&  \mathbf{b}^T\\
\end{array}
=
\begin{array}
{c|ccc}
c_1 & a_{11} & \cdots & a_{1s}  \\
\vdots & \vdots  & & \vdots \\
c_s & a_{s1} & \cdots & a_{ss} \\
\hline
& b_1 &  \cdots & b_s \\
\end{array},
\]
where $\mathbf{A} \in \mathbb{R}^{s,s}$, $\mathbf{b} \in \mathbb{R}^s$, and $\mathbf{c}=\mathbf{A} \mathbf{l}$, with $\mathbf{l}=(1,1,\cdots,1)^T \in \mathbb{R}^s$. Applying an $s$-stage collocation method to \eqref{eq:DE}, we obtain the following scheme.
\begin{scheme} [$s$-stage HEQ Collocation Method] \label{eq:Collocation-Scheme}
Let $c_1,\cdots, c_s$ be distinct real numbers ($0 \leq c_i \leq 1$). Given $\Psi^n$, the collocation polynomial $u(t)$ is a polynomial of degree $s$ satisfying
\beq
\bea{l}
u(t_n) = \Psi^n, \\\\
\dot{u}(t_n+c_i \Delta t) = \cN \big(u(t_n+c_i \Delta t)\big) \cB \big(u(t_n+c_i \Delta t)\big), \quad i=1,\cdots, s,
\eea
\eeq
and the numerical solution is defined by $\Psi^{n+1} = u(t_n+\Delta t)$.
\end{scheme}
Theorem 1.4 on page 31 of \cite{HairerBook} indicates that the collocation method yields a special RK method. If the collocation points $c_1,\cdots,c_s$ are chosen as the Gaussian quadrature nodes, i.e., the zeros of the $s$-th shifted Legendre polynomial $\frac{d^s}{dx^s} \Big(  x^s (x-1)^s \Big),$  Scheme \ref{eq:Collocation-Scheme} is called the Gaussian collocation method. Based on the Gaussian quadrature nodes, the interpolating quadrature formula has order $2s$, and the Gaussian collocation method shares the same order $2s$. Collocation points for Gaussian collocation methods of order 4 and 6 are given explicitly in \cite{HairerBook}.

For conservative systems with quadratic invariants, the Gaussian collocation methods have been proven to conserve the corresponding discrete quadratic invariants. For more details, please refer to the book \cite{HairerBook}. Applying the theory to our reformulated EQ system \eqref{eq:gradient-flow-EQ-full}, we have the following theorem.
\begin{thm} \label{thm:High-EQ}
The $s$-stage HEQ Gaussian collocation Scheme \ref{eq:Collocation-Scheme} is unconditionally energy stable, i.e., it satisfies the following energy dissipation law
\beq\label{eq:High-EL}
F^{n+1}-F^n = \Delta t \sum_{i=1}^s b_i \Big(  \cB u(t_n+c_i \Delta t), \cN\big(u(t_n+c_i \Delta t)\big) \cB u(t_n+c_i \Delta t) \Big)\leq 0,
\eeq
where $F^{n} = \frac{1}{2}\| \Psi^{n}\|_{\cB}^2 - A,$ $c_i$ ($i=1,\cdots,s$) is the Gaussian quadrature nodes, $b_i\geq 0$ ($i=1,\cdots,s$) are the Gauss-Legendre quadrature weights, $u(t)$ be the collocation polynomial of the Gaussian collocation methods.
\end{thm}

\begin{proof}
Denoting $\Psi^n = u(t_n)$ and $\Psi^{n+1} = u(t_{n+1}),$ we have
\ben
F^{n+1}-F^n &=& \frac{1}{2}\| \Psi^{n+1}\|_{\cB}^2 - \frac{1}{2}\| \Psi^{n}\|_{\cB}^2 = \frac{1}{2}\| u(t_{n+1})\|_{\cB}^2 - \frac{1}{2}\| u(t_{n})\|_{\cB}^2\nonumber\\
&=& \int_{t_n}^{t_{n+1}} \frac{1}{2}\frac{d}{d t}\| u(t)\|_{\cB}^2 dt = \int_{t_n}^{t_{n+1}} \Big(\dot{u}(t),\cB u(t)\Big)d t,
\een
where the self-adjoint property of $\cB$ was used. The integrand $\Big(\dot{u}(t),\cB u(t)\Big)$ is a polynomial of degree $2s-1$, which is integrated without error by the $s$-stage Gaussian quadrature formula. It therefore follows from the collocation condition that
\beq
\int_{t_n}^{t_{n+1}} \Big(\dot{u}(t),\cB u(t)\Big)d t = \Delta t \sum_{i=1}^s b_i \Big(  \cB u(t_n+c_i \Delta t), \cN\big(u(t_n+c_i \Delta t)\big) \cB u(t_n+c_i \Delta t) \Big),
\eeq
which leads to \eqref{eq:High-EL}. This completes the proof.
\end{proof}

For general RK methods, we have the following theorem.
\begin{thm} \label{thm:High-EQ-RK}
If the coefficients of a HEQ-RK method satisfy
\beq\label{eq:RK-stable-condition}
b_i a_{i j} + b_j a_{j i} = b_i b_j,\quad b_i\geq 0,\quad \forall~i,j = 1,\cdots,s,
\eeq
then it is unconditionally energy stable, i.e., it satisfies the following energy dissipation law
\beq\label{eq:High-EL-RK}
F^{n+1}-F^n = \Delta t\sum_{i=1}^s b_i \Big(  \cB \Psi_i, \cN(\Psi_i) \cB \Psi_i \Big)\leq 0,
\eeq
where $F^{n} = \frac{1}{2}\| \Psi^{n}\|_{\cB}^2 - A,$ $\Psi_i = \Psi^{n}+\Delta t\sum\limits_{j=1}^s a_{ij} k_j.$
\end{thm}

\begin{proof}
Denoting $\Psi^{n+1}  =  \Psi^n  + \Delta t \sum\limits_{i=1}^s b_i k_i$ and noticing that operator $\cB$ is linear and self-adjoint, we have
\beq
\frac{1}{2}\| \Psi^{n+1}\|_{\cB}^2 = \frac{1}{2}\| \Psi^{n}\|_{\cB}^2 + \Delta t\sum\limits_{i=1}^s b_i(k_i,\cB\Psi^{n}) +\frac{\Delta t ^2}{2}\sum\limits_{i,j=1}^s b_i b_j (k_i,\cB k_j),
\eeq
which implies
\beq\label{eq:RK-ES-t-1}
F^{n+1} - F^n = \Delta t\sum\limits_{i=1}^s b_i(k_i,\cB\Psi^{n}) +\frac{\Delta t ^2}{2}\sum\limits_{i,j=1}^s b_i b_j (k_i,\cB k_j).
\eeq
Applying $\Psi^{n} = \Psi_i - \Delta t\sum\limits_{j=1}^s a_{ij} k_j$ to \eqref{eq:RK-ES-t-1}, we obtain
\beq\label{eq:RK-ES-t-2}
F^{n+1} - F^n = \Delta t\sum\limits_{i=1}^s b_i(k_i,\cB\Psi_i) +\frac{\Delta t ^2}{2}\sum\limits_{i,j=1}^s (b_i b_j - b_i a_{i j} - b_j a_{j i}) (k_i,\cB k_j).
\eeq
Combining \eqref{eq:RK-stable-condition}, \eqref{eq:RK-ES-t-2} and $k_i = \cN(\Psi_i)\cB\Psi_i$, we obtain \eqref{eq:High-EL-RK}. This completes the proof.
\end{proof}

\begin{rem}
Due to the collocation method reduces to a special RK method, we have to solve the nonlinear system \eqref{HEQRK}, which will be implemented by using the following simple fixed-point iteration method. Denote $\cN(\Psi) = \cN_1 + \cN_2(\Psi)$, where $\cN_1$ and $\cN_2(\Psi)$ are the linear part and nonlinear part of $\cN(\Psi)$, respectively. At time step $n$, the nonlinear system for $k_i$ in \eqref{HEQRK} is first computed by
\beq\label{interation-eq}
k_i^{r+1} =  \cN_1 \cB \Big( \Psi^n+ \Delta t \sum\limits_{j=1}^s a_{ij} k_j^{r+1} \Big) + \cN_2 \Big(\Psi^n+ \Delta t \sum\limits_{j=1}^s a_{ij} k_j^{r} \Big) \cB \Big( \Psi^n+ \Delta t \sum\limits_{j=1}^s a_{ij} k_j^{r} \Big), \quad i=1,\cdots,s,
\eeq
where we take the initial iteration $k_i^0 = 0$ for simplicity. We iterate the solution until the following criteria is satisfied
\beq
\max_{i}\|k_i^{r+1}-k_i^r\|_{\infty} < 10^{-12}.
\eeq
We note that the FFT algorithm is applied for solving the linear equation system \eqref{interation-eq}. Then we obtain $\Psi^{n+1}$ by \eqref{HEQRK}.
\end{rem}

\begin{rem}
Though the energy \eqref{initial-energy} and \eqref{EQ-energy} are equivalent in the continuum form, the  proposed schemes only satisfy a discrete energy dissipation law in term of the reformulated energy \eqref{EQ-energy}, instead of the original energy \eqref{initial-energy}. However, we point out the discrete version of \eqref{EQ-energy} is a high-order approximation of \eqref{EQ-energy}, i.e. \eqref{initial-energy}. In addition, the stabilization technique is introduced to inherit the stability of the original energy as much as possible. For instance, in the Ginzburg-Landau free energy \eqref{Ginzburg-Landau-energy}, the modified energy is $F = \frac{\varepsilon^2}{2}\|\nabla \phi\|^2 + \frac{\gamma_0}{2}\|\phi\|^2 + \frac{1}{2}\|q\|^2 - A$, which implies $H^1$ boundedness of phase variable $\phi$. More details, please refer to \cite{Chen&Zhao&Yang2018,SAV-2}.
\end{rem}

\begin{rem}
We note that the EQ approach is employed for developing second-order linear or arbitrarily high-order schemes in this paper, for respecting the energy-dissipation law. For the gradient flow models equipped with singular energy potentials, such as the Cahn-Hilliard equation with the Flory-Huggins free energy, we can't prove that the proposed methods preserve the positivity. For some seminal work of the positive-preserving algorithms for gradient flow models, please refer to \cite{ChenetalPositivity}.
\end{rem}

\begin{rem}
For gradient flow models,  a class of high-order convex splitting Runge-Kutta schemes have been developed recently \cite{RKJCP2017}. However, these schemes only work when one can find a convex splitting of the free energy, which may not always be true. Besides, they require too many multi-stages to reach even the 3rd order accuracy. The existence of higher order convex splitting RK coefficients is not guaranteed. Moreover, they don't work for the case where the mobility is variable. The HEQ schemes introduced in this paper do not have these constraints so that they can be applied to a much broader class of problems.
\end{rem}

\section{Spatial discretization}\label{sec:spaceDis}
Next, we present structure-preserving spacial discretizations for the quadratized gradient flow models. For the spatial discretization, one idea inspired by \cite{Dahlby2011A} is to preserve the negative semi-definite property of operator $\cN(\Psi)$ and the self-adjoint, positive definite property of operator $\cB$. Another idea is to develop spatial discretization methods that preserve the discrete integration-by-parts formulae (please see \cite{Gong2016Fully, Gong&Zhao&WangACM, Gong2018Second, Gong&Zhao&Yang&WangSISC, Gong&Zhao&WangCPC, Zhao&Yang&Gong&WangCMAME2017} for details). Based on these ideas, we apply the Fourier pseudospectral method in space to \eqref{eq:gradient-flow-EQ-full}, which leads to an ODE system that preserves the spatial semi-discrete energy dissipation law. Then we apply the methods discussed in the previous section to the ODE system to obtain fully discrete energy stable schemes.

Let $N_x,N_y$ be two positive even integers. The spatial domain $\Omega = [0,L_x]\times[0,L_y]$ is uniformly partitioned with mesh size $h_x =
L_{x}/N_{x},h_y = L_{y}/N_{y}$ and $$\Omega_{h} =
\left\{(x_{j},y_{k})|x_{j} = j h_x,y_{k} = kh_y,~0\leq j\leq
N_{x}-1,0\leq k\leq N_{y}-1\right\}.$$
We define
$S_{N} = \textrm{span}\{ X_{j}(x) Y_{k}(y), j =
0,1,\ldots,N_{x}-1; k = 0,1,\ldots,N_{y}-1\}$ as the interpolation space, where $X_{j}(x)$ and $Y_{k}(y)$ are
trigonometric polynomials of degree $N_{x}/2$ and $N_{y}/2$, given
respectively by
\beq
X_{j}(x) = \frac{1}{N_{x}}\sum\limits_{m =
-N_{x}/2}^{N_{x}/2}{\frac{1}{a_{m}}e^{im\mu_{x}(x-x_{j})}},\quad
Y_{k}(y) = \frac{1}{N_{y}}\sum\limits_{m =
-N_{y}/2}^{N_{y}/2}{\frac{1}{b_{m}}e^{im\mu_{y}(y-y_{k})}},
\eeq
where $a_{m} =
\begin{cases}
1, |m|<N_{x}/2,\\
2, |m|=N_{x}/2,
\end{cases} \mu_{x} = 2\pi/L_{x},~
b_{m} =
\begin{cases}
1, |m|<N_{y}/2,\\
2, |m|=N_{y}/2,
\end{cases} \mu_{y} = 2\pi/L_{y}.
$ We define the interpolation operator $I_{N}:C(\Omega)\rightarrow
S_{N}$ as follows:
\begin{equation}\label{interpolation}
I_{N}u(x,y) = \sum\limits_{j = 0}^{N_{x}-1}\sum\limits_{k
= 0}^{N_{y}-1} u_{j,k} X_{j}(x) Y_{k}(y),
\end{equation}
where $u_{j,k} = u(x_{j},y_{k})$. The key of spatial Fourier pseudospectral discretization is to obtain derivative $\partial_{x}^{s_{1}}\partial_{y}^{s_{2}}I_{N}u(x,y)$ at
collocation points. Then, we differentiate \eqref{interpolation} and evaluate the resulting expressions at point $(x_{j},y_{k})$ as follows
\begin{equation*}
\partial_{x}^{s_{1}}\partial_{y}^{s_{2}}I_{N}u(x_{j},y_{k}) = \sum\limits_{m_1 =
0}^{N_{x}-1}\sum\limits_{m_2
= 0}^{N_{y}-1} u_{m_1,m_2} (\bD_{s_{1}}^{x})_{j,m_1}
(\bD_{s_{2}}^{y})_{k,m_2},
\end{equation*}
where $\bD_{s_{1}}^{x}$ and $\bD_{s_{2}}^{y}$ are $N_{x} \times N_{x}$ and $N_{y} \times N_{y}$ matrices, respectively, with elements given by
\begin{equation*}
(\bD_{s_{1}}^{x})_{j,m} = \frac{d^{s_{1}}X_{m}(x_{j})}{dx^{s_{1}}},~(\bD_{s_{2}}^{y})_{k,m}
=\frac{d^{s_{2}}Y_{m}(y_{k})}{dy^{s_{2}}}.
\end{equation*}
Here we note that the Fourier pseudospectral method preserves discrete integration-by-parts formulae. For more details, please refer to our previous work \cite{Gong&Zhao&WangACM}.

Applying the Fourier pseudospectral method to \eqref{eq:gradient-flow-EQ-full}, we obtain
\beq \label{spaceDis}
\frac{d}{d t} \Psi = \cN_d(\Psi) \cB_d \Psi,
\eeq
where
$\cN_d(\Psi)$ is a discrete negative semi-definite operate that approximates to $\cN(\Psi)$, $\cB_d$ is a constant self-adjoint, positive definite operate that approximates to $\cB.$ It is readily to show that the system \eqref{spaceDis} possesses the discrete energy dissipation law
\beq
\frac{d}{d t} F_h = \left( \cB_d \Psi , \frac{d}{d t}\Psi\right)_h = \big(\cB_d \Psi , \cN_d(\Psi)  \cB_d\Psi \big)_h \leq 0,
\eeq
where $F_h = \frac{1}{2}\left(\Psi, \cB_d \Psi\right)_h-A,$ $(\cdot,\cdot)_h$ is the corresponding discrete inner product. Then the linear prediction-correction schemes and Gaussian collocation methods proposed in Section \ref{sec:timeDis} can be applied directly for \eqref{spaceDis} to obtain fully discrete energy stable schemes.

Next, we specifically apply our spatial discretization method to two examples. Firstly, we consider the following Cahn-Hilliard equation
\beq\label{CH}
\partial_t \phi = \lambda\Delta\frac{\delta F}{\delta\phi},
\eeq
where the free energy functional $F$ is given by
\beq\label{CH-energy}
F = \left(\frac{1}{4}\phi^4 - \frac{a}{2}\phi^2,1\right) + \frac{b}{2}\|\nabla\phi\|^2 + \frac{c}{2}\big(\|\phi\|^2 - 2\|\nabla\phi\|^2 + \|\Delta\phi\|^2\big).
\eeq
By introducing $q = \frac{1}{\sqrt{2}}(\phi^2 - a - \gamma_0),$ the free energy can be rewritten as
\beq\label{CH-EQenergy}
F = \frac{1}{2}\Big(\phi, \cL\phi\Big) + \frac{1}{2}\|q\|^2 - A + \textrm{boundary terms},
\eeq
where $\cL = \gamma_0 - b\Delta + c(1+\Delta)^2$ and $A = \frac{1}{4}(a+\gamma_0)^2|\Omega|.$ Assume the boundary conditions annihilate the boundary terms. Then Eq. \eqref{CH} reduces to the following system
\beq\label{CH-EQ-form}
\begin{cases}
\partial_t\phi = \lambda\Delta(\cL\phi+\sqrt{2}\phi q),\\
\partial_t q = \sqrt{2}\phi \partial_t\phi,
\end{cases}
\eeq
which can be written into the compact form \eqref{eq:gradient-flow-EQ-full} with $\Psi = (\phi, q)^T,$ $\cN(\Psi) = (1, \sqrt{2}\phi)^T \lambda \Delta (1, \sqrt{2}\phi)$ and $\cB = \textrm{diag}(\cL,1).$ Following the notations of Ref. \cite{Gong&Zhao&WangACM} and applying the Fourier pseudospectral method to \eqref{CH-EQ-form}, we obtain
\beq\label{CH-spaceDis}
\begin{cases}
\frac{d}{dt}\phi = \lambda\Delta_h(\cL_d \phi + \sqrt{2}\phi \odot q),\\
\frac{d}{dt} q = \sqrt{2}\phi \odot \frac{d}{dt}\phi,
\end{cases}
\eeq
where $\Delta_h = \bD_{2}^{x}\textcircled{x} + \bD_{2}^{y}\textcircled{y}$ and $\cL_d = \gamma_0 - b\Delta_h + c(1+\Delta_h)^2.$ System \eqref{CH-spaceDis} can be written into the form of \eqref{spaceDis} with discrete operators $\cN_d(\Psi) = (1, \sqrt{2}\phi\odot)^T \lambda \Delta_h (1, \sqrt{2}\phi\odot)$ and $\cB_d = \textrm{diag}(\cL_d,1).$

Next we consider the molecular beam epitaxy model
\beq\label{MBE}
\partial_t \phi = -\lambda\frac{\delta F}{\delta\phi},
\eeq
where the free energy functional $F$ is given by
\beq\label{MBE-energy}
F = \frac{\varepsilon^2}{2}\|\Delta\phi\|^2 + \frac{1}{4}\big\| |\nabla\phi|^2-1 \big\|^2.
\eeq
Letting $q = \frac{1}{\sqrt{2}}\big(|\nabla\phi|^2-1 - \gamma_0\big),$ we obtain the corresponding free energy
\beq\label{MBE-EQenergy}
F = \frac{1}{2}\Big(\phi, \cL\phi\Big) + \frac{1}{2}\|q\|^2 - A + \textrm{boundary terms},
\eeq
where $\cL = \varepsilon^2\Delta^2 - \gamma_0\Delta$ and $A = \frac{1}{4}(2\gamma_0+\gamma_0^2)|\Omega|.$ Assume the boundary conditions annihilates the boundary term. Then Eq. \eqref{MBE} is written equivalently as
\beq\label{MBE-EQ-form}
\begin{cases}
\partial_t\phi = -\lambda\Big(\cL\phi-\sqrt{2}\nabla\cdot(q\nabla\phi)\Big),\\
\partial_t q = \sqrt{2}\nabla\phi\cdot \nabla\partial_t\phi,
\end{cases}
\eeq
which can be written into the compact form of \eqref{eq:gradient-flow-EQ-full} with $\Psi = (\phi, q)^T,$ $\cN(\Psi) = -\lambda \left(\begin{array}{c}
1 \\
\sqrt{2}\nabla\phi\cdot\nabla
\end{array}
\right) (1, -\nabla\cdot \sqrt{2}\nabla\phi)$ and $\cB = \textrm{diag}(\cL,1).$
Applying the Fourier pseudospectral method for \eqref{MBE-EQ-form}, we obtain
\beq\label{MBE-spaceDis}
\begin{cases}
\frac{d}{dt} \phi = -\lambda\Big(\cL_d\phi - \sqrt{2}\bD_{1}^{x}\textcircled{x}\big(q\odot\bD_{1}^{x}\textcircled{x}\phi\big) - \sqrt{2}\bD_{1}^{y}\textcircled{y}\big(q\odot\bD_{1}^{y}\textcircled{y}\phi\big)\Big),\\
\frac{d}{dt} q = \sqrt{2} (\bD_{1}^{x}\textcircled{x}\phi)\odot (\bD_{1}^{x}\textcircled{x}\frac{d}{dt}\phi) + \sqrt{2} (\bD_{1}^{y}\textcircled{y}\phi)\odot (\bD_{1}^{y}\textcircled{y}\frac{d}{dt}\phi),
\end{cases}
\eeq
where $\cL_d = \varepsilon^2\Delta_h^2 - \gamma_0\Delta_h.$ System \eqref{MBE-spaceDis} can be written into the form of \eqref{spaceDis} with discrete operators $\cN_d(\Psi) = -\lambda
\left(\begin{array}{c}
1 \\
\sqrt{2} (\bD_{1}^{x}\textcircled{x}\phi)\odot \bD_{1}^{x}\textcircled{x} + \sqrt{2} (\bD_{1}^{y}\textcircled{y}\phi)\odot \bD_{1}^{y}\textcircled{y}
\end{array}
\right) \Big(1, -\sqrt{2}\bD_{1}^{x}\textcircled{x}\big(\bD_{1}^{x}\textcircled{x}\phi\big)\odot - \sqrt{2}\bD_{1}^{y}\textcircled{y}\big(\bD_{1}^{y}\textcircled{y}\phi\big)\odot\Big)$ and $\cB_d = \textrm{diag}(\cL_d,1).$

\section{Numerical Results}\label{sec:Numer}
In this section, we apply the proposed numerical schemes to several gradient flow problems.  For simplicity, we assume periodic boundary conditions in the numerical experiments. Note that our schemes can be easily applied to gradient flow models with physical boundary conditions as long as they annihilate the boundary terms in the energy dissipation rate formula.  Systematical comparisons among the prediction-correction scheme \ref{scheme:PC}, the BDF-k scheme \ref{scheme:GradientFlow-BDFK}, and HEQ schemes \ref{scheme:RK-method} and \ref{eq:Collocation-Scheme} are presented below.

We fix the number of iteration $N=5$ and $\varepsilon_0=10^{-12}$ in the rest of this paper. In addition, for easily referring the schemes, we name \eqref{CN-PC} with  predictor \eqref{CN-BDF-P0} as the LCN  (linear CN) scheme; \eqref{CN-PC} with  predictor \eqref{CN-P2} as the ICN (implicit CN) scheme; \eqref{BDF2-PC} with  predictor \eqref{CN-BDF-P0} as the LBDF2 scheme; \eqref{BDF2-PC} with predictor \eqref{BDF2-P2} as the IBDF2 scheme.


\textbf{Example 1: Cahn-Hilliard equation.} In the first example, we consider the Cahn-Hilliard equation with the  free energy functional containing a double-well bulk term:
\beq
F = \frac{\varepsilon^2}{2}\|\nabla \phi\|^2 + \frac{1}{4}\|\phi^2-1\|^2,
\eeq
where $\varepsilon$ is a small parameter.
The Cahn-Hilliard equation is given  as follows
\beq
\partial_t \phi = \lambda \Delta \Big( -\varepsilon^2 \Delta \phi + \phi^3 - \phi \Big),
\eeq
where $\lambda$ is the mobility parameter.
If we introduce the auxiliary variable $q =\frac{1}{\sqrt{2}} (\phi^2 -1 - \gamma_0)$, where $\gamma_0>0$ is a constant, we have
\beq
\left\{
\bea{l}
\partial_t \phi = \lambda \Delta ( -\varepsilon^2 \Delta \phi + \gamma_0 \phi + q g(\phi) ), \\
\partial_t q = g(\phi) \partial_t \phi, \quad g(\phi) = \sqrt{2} \phi.
\eea
\right.
\eeq
Then we denote $\Psi=(\phi, q)^T$ and rewrite the Cahn-Hilliard equation into a prototypical form
\beq
\partial_t \Psi = \cA^T \cG \cA \cB \Psi, \quad \cA = \big(1 , g(\phi)\big), \quad \cB = \textrm{diag}(-\varepsilon^2 \Delta +\gamma_0, 1), \quad \cG= \lambda \Delta.
\eeq

In the numerical experiment,  we choose a square domain with $\Omega =[0, 1]^2$, and the parameters are chosen as $\gamma_0=1$, $\varepsilon=0.01$ and $\lambda = 10^{-3}$.   Using   initial condition $\phi(x,y) = \sin (2\pi x) \sin (2\pi y)$, we carry out time-step refinement tests. By choosing the numerical solution calculated by the $6$th order Gaussian collocation method with time step $\Delta t = 10^{-5}$ as the ``exact"  solution, the errors are calculated as the differences between the numerical solutions at different time steps and the ``exact" solution, respectively.

The $\log$-$\log$ plots of $L^2$ errors against time step $\Delta t$ are depicted in Figure \ref{fig:CH-Error}. From Figure \ref{fig:CH-Error}(a), we observe that the HEQ-RK scheme \ref{scheme:RK-method} with the 4th and 6th order Gaussian collocation points reaches the 4th and 6th order of accuracy,  respectively.  From Figure \ref{fig:CH-Error}(b), we observe that both the LCN scheme and the ICN scheme reach the second order accuracy, while the ICN scheme reaches it at a larger time step. The ICN scheme has a significantly smaller error than the LCN scheme by comparing the errors of the two schemes at the same time step size, with similar CPU time cost, since the ICN/IBDF2 schemes only require up to 5 iterations, where each iteration is solved using an FFT solver. As shown in Figure \ref{fig:CH-Error}(c), the BDF schemes have similar results as the CN schemes do. Besides, the BDF-k schemes also reach the expected order of convergence when the time step is small enough, as shown in Figure \ref{fig:CH-Error}(d). Comparing between Figure \ref{fig:CH-Error}(a) and \ref{fig:CH-Error}(d), we observe the HEQ schemes are more accurate than the BDF-k schemes when using the same time steps.

\begin{figure}
\centering
\subfigure[HEQ-RK schemes]{\includegraphics[height=3.75cm,width=3.75cm,angle=0]{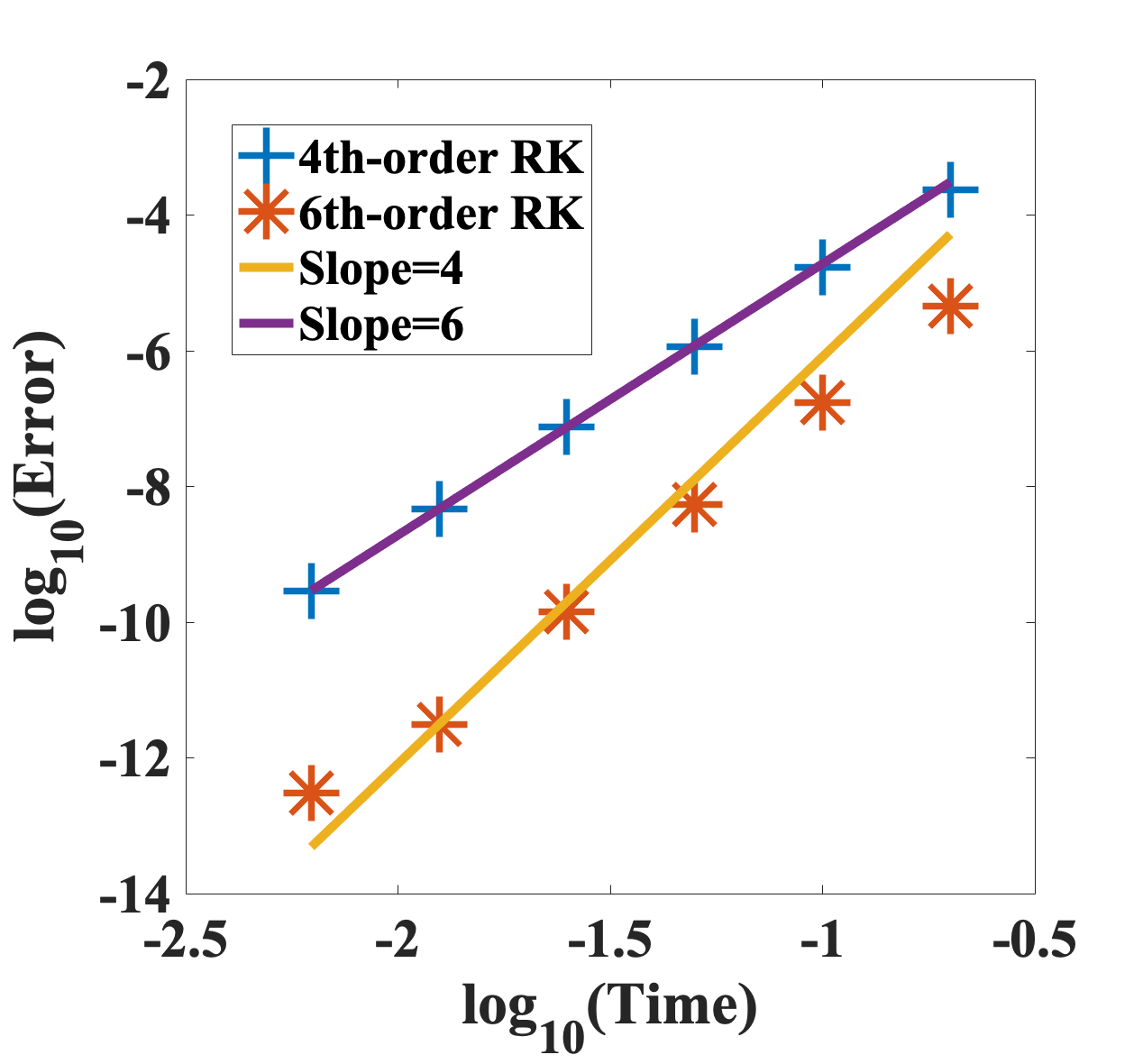}}
\subfigure[CN schemes]{\includegraphics[height=3.75cm,width=3.75cm,angle=0]{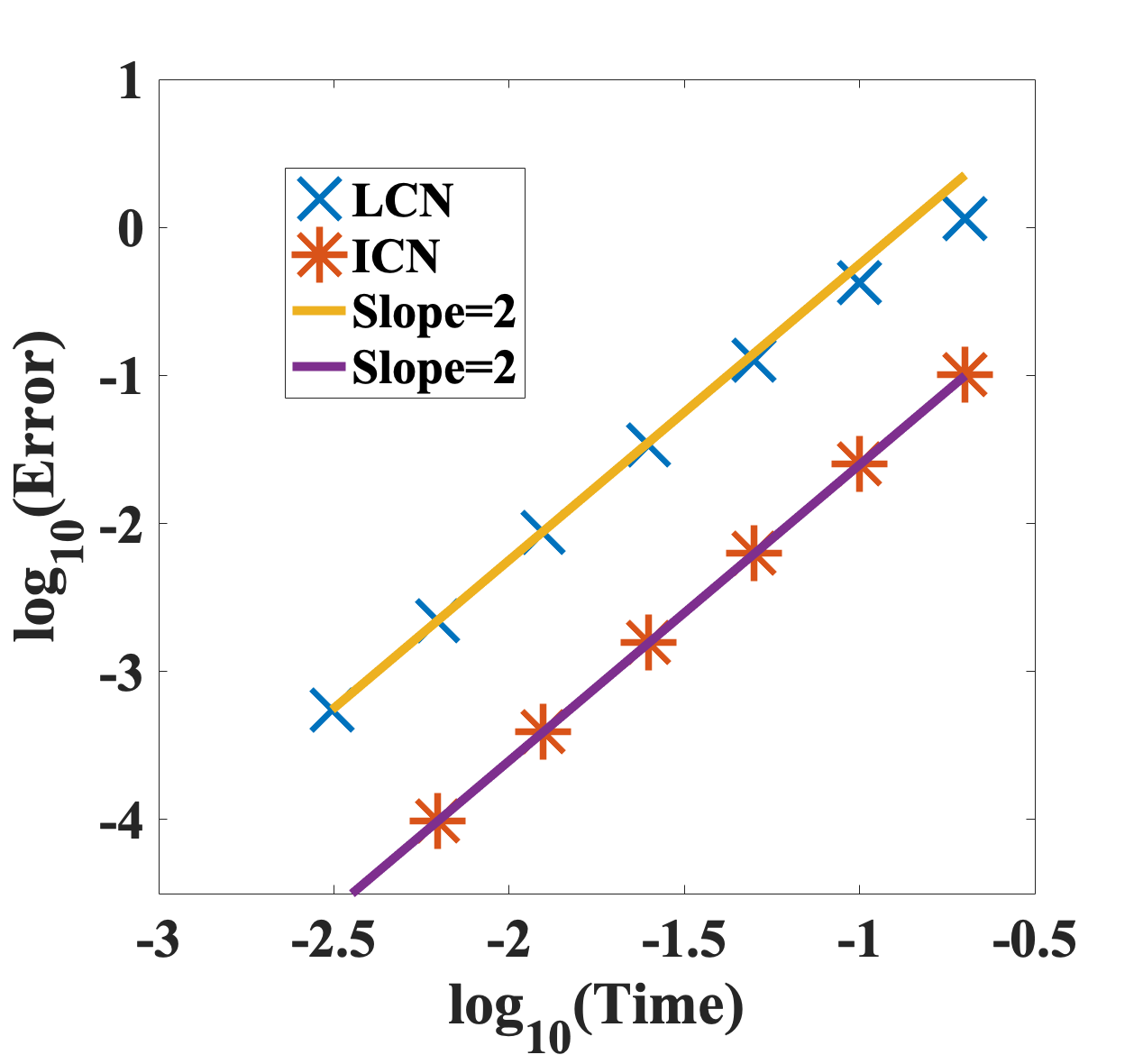}}
\subfigure[BDF2 schemes]{\includegraphics[height=3.75cm,width=3.75cm,angle=0]{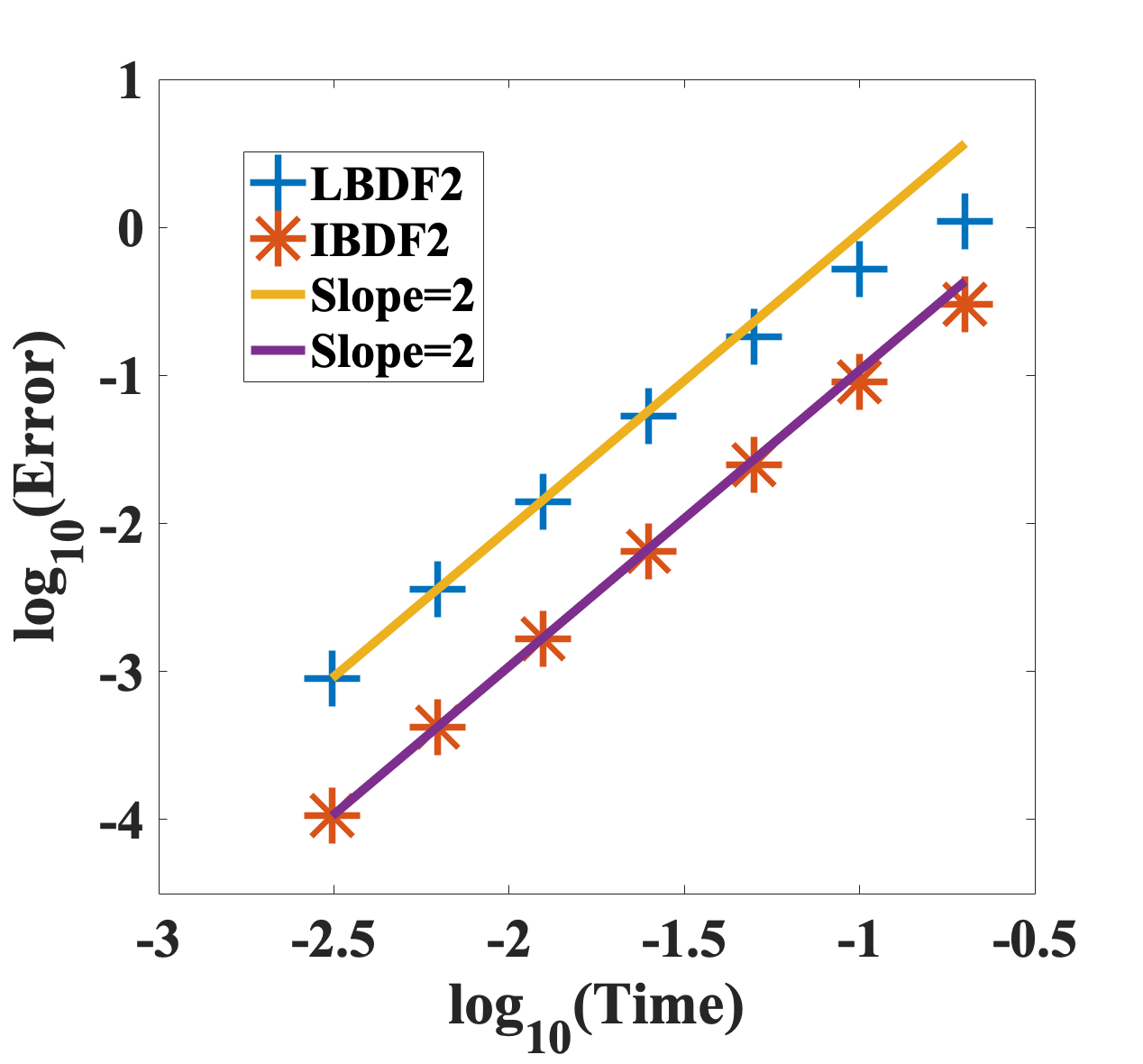}}
\subfigure[High-order BDF schemes]{\includegraphics[height=3.75cm,width=3.75cm,angle=0]{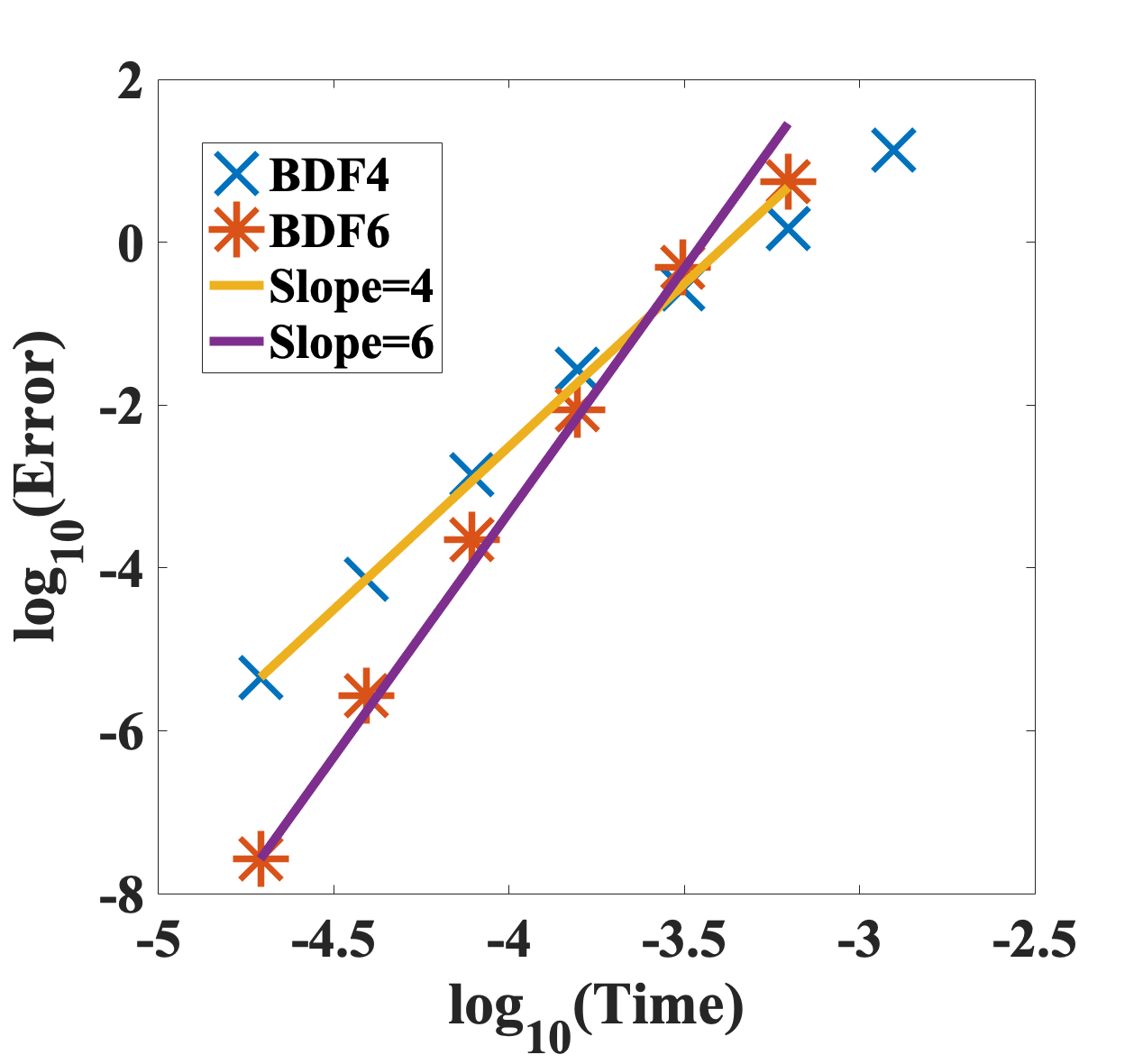}}
\caption{Rate of convergence in time. This figure shows some time-step refinement tests.  It provides the $\log$-$\log$ plots of numerical $L^2$ errors with respect to different time steps using  the RK or Gaussian collocation method, CN, BDF2, BDF4, BDF6 schemes. Here we use spatial meshes $256 \times 256$ and the error is calculated at $t=1$. It demonstrates the proposed schemes reach their expected order accuracy. }
\label{fig:CH-Error}
\end{figure}

Next, we study the coarsening dynamics, and compare the accuracy of the proposed schemes by comparing the numerically calculated energies using different schemes with various time-step sizes. To have a more detailed comparison, we consider coarsening dynamics of a binary mixture (which demonstrates a dramatic energy change when the system coarsens). We use $256 \times 256$ equal distanced meshes in space to discretize the domain $\Omega=[0 \,\,\, 4\pi] \times [0 \,\,\, 4\pi]$. The parameters are chosen as $\lambda=0.1$, $\epsilon=0.05$, $\gamma_0=1$, and initial condition is $\phi(x,y) = 0.001\big(2~ \textrm{rand}(x,y)-1\big)$.

In Figure \ref{fig:Coarsening-CN-Energy}(a-b), the energy computed using the CN schemes in time $[0 \,\,\, 2]$  are plotted. We observe that for the LCN scheme, when $\Delta t < 0.0125$, the calculated energy converges to the accurate energy. For ICN scheme, it has a dramatic improvement in accuracy, where the numerical energy calculated with $\Delta t=0.1$ is already very accurate, where the CPU time cost is negligible as only up to 5 times of an FFT solver is applied to  \eqref{CN-P2} in the computation in each step. As shown in Figure \ref{fig:Coarsening-CN-Energy}(c-d), the BDF4 and BDF6 schemes provide accurate results given the time step is small enough. However, when the time step is large, it violates the energy dissipation law. The results using the HEQ Gaussian collocation method of order four and order six  are summarized in Figure \ref{fig:Coarsening-CN-Energy}(e-f).  We observe that the HEQ Gaussian collocation scheme provides much better accuracy. Both 4th order and the 6th order schemes with a large time step size, such as $\Delta t  = 0.2$, give very accurate energy predictions, which is much larger than the proper time steps for the BDF4 and BDF6 schemes respectively.

\begin{figure}
\centering
\subfigure[LCN Scheme]{\includegraphics[height=4cm]{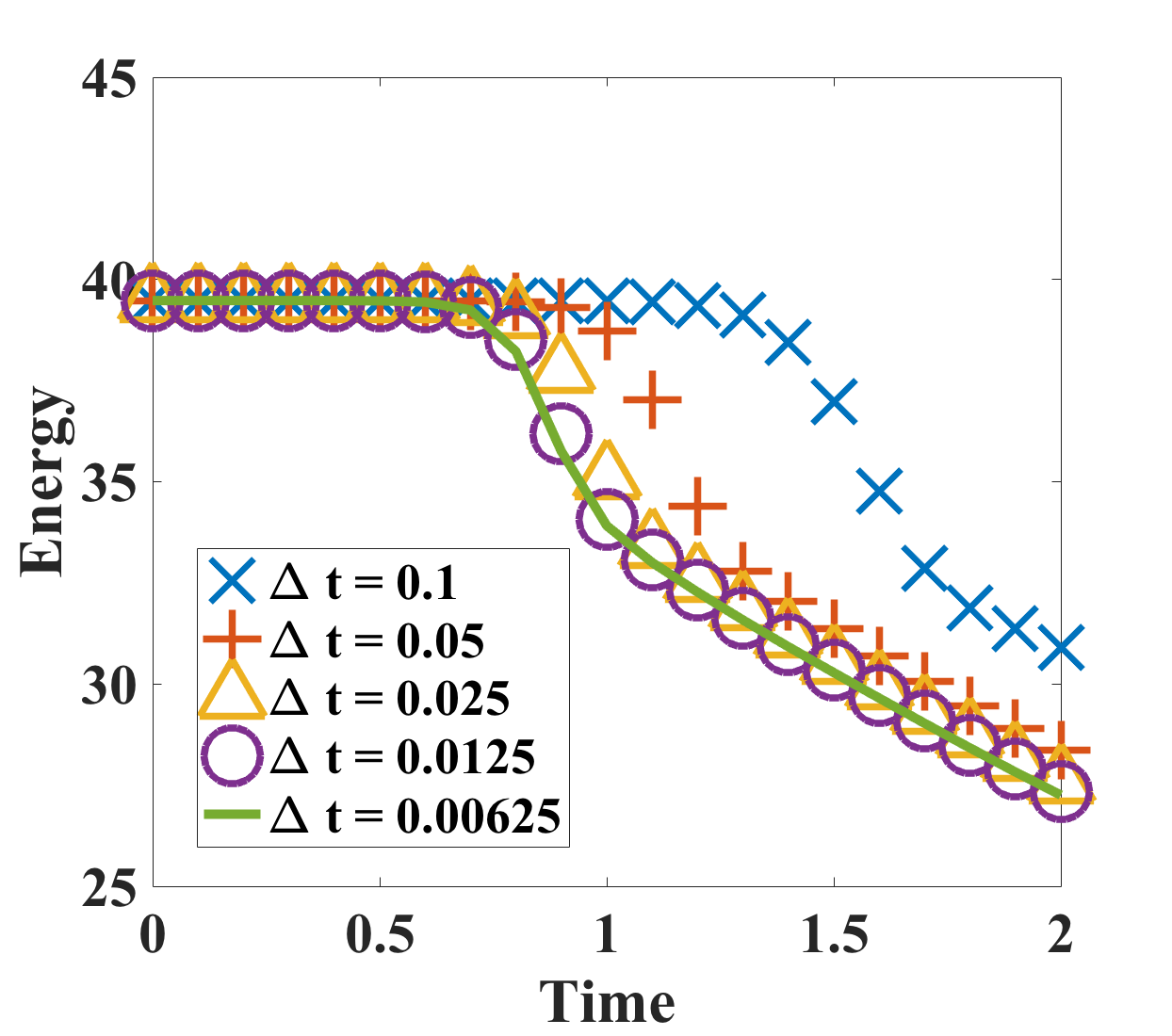}}
\subfigure[ICN Scheme]{\includegraphics[height=4cm]{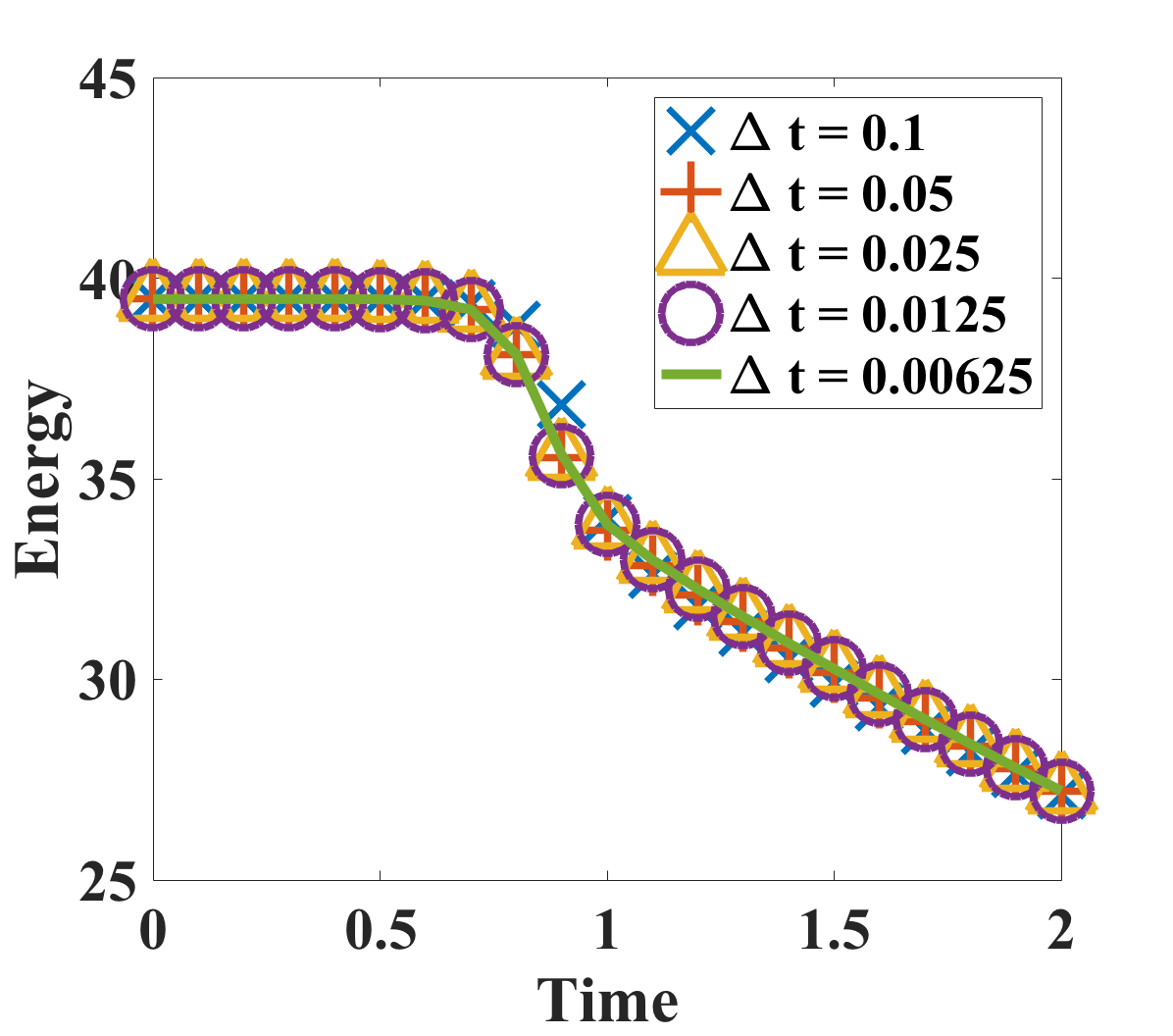}}
\subfigure[BDF4 Scheme]{\includegraphics[height=4cm]{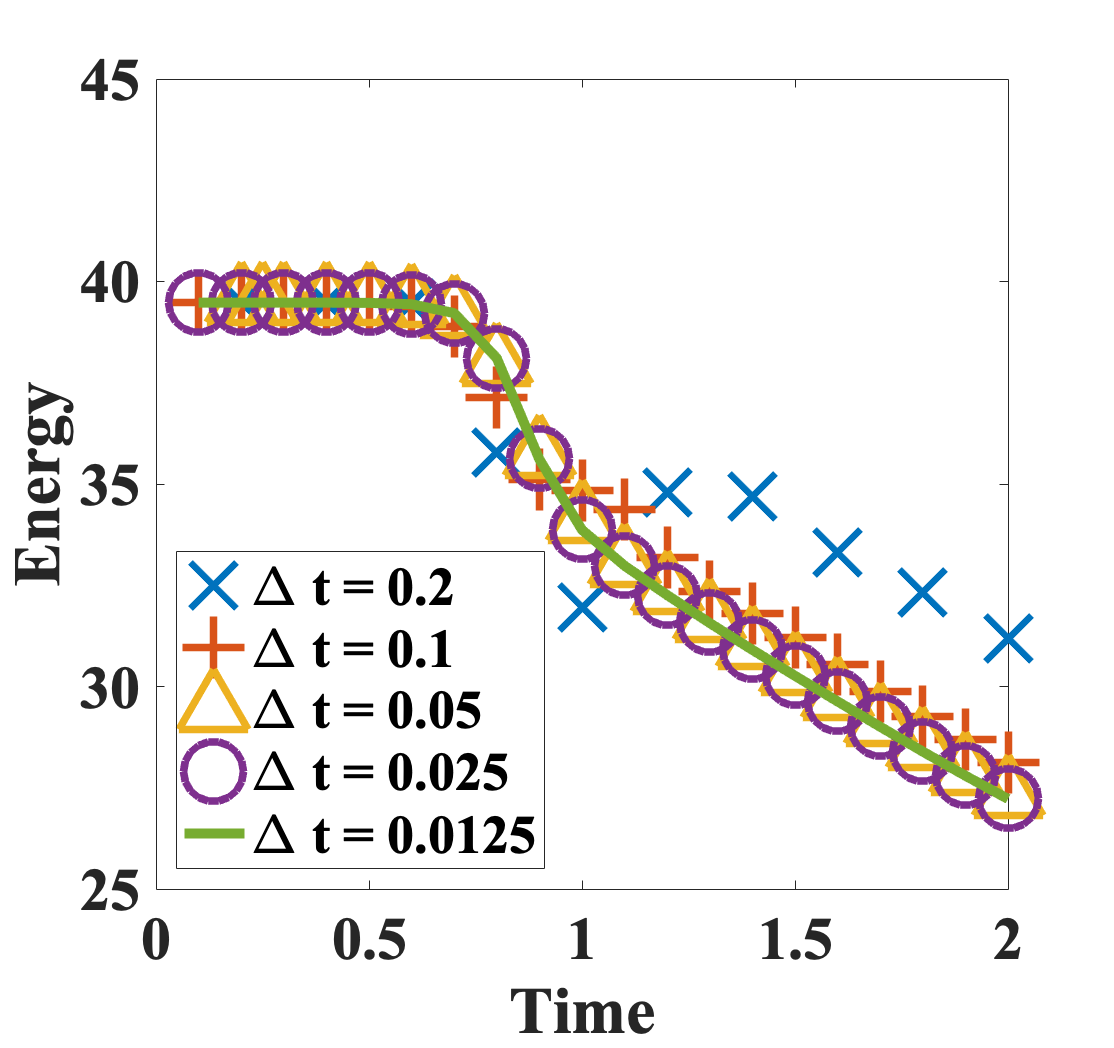}}

\subfigure[BDF6 Scheme]{\includegraphics[height=4cm]{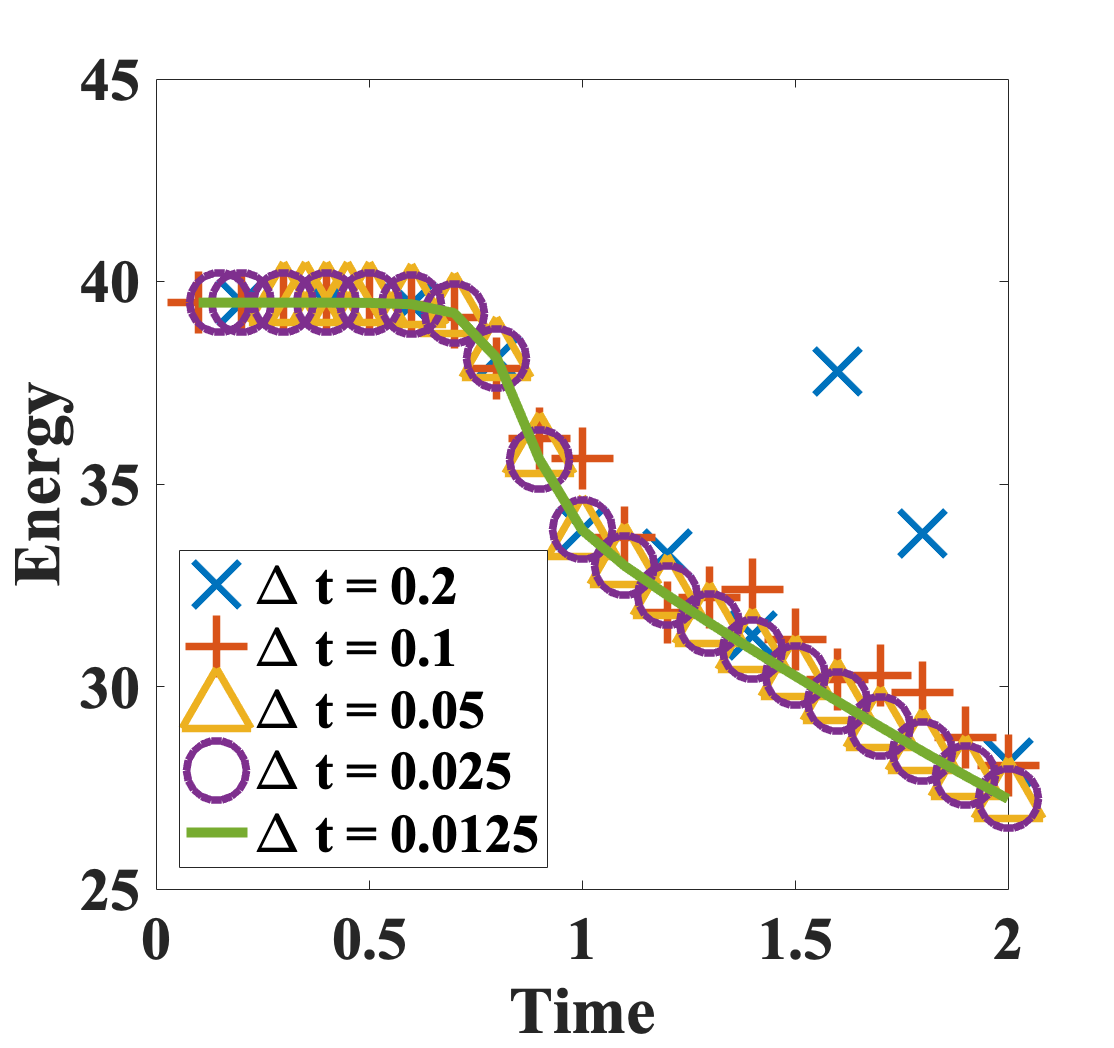}}
\subfigure[4th Order Gauss Method]{\includegraphics[height=4cm]{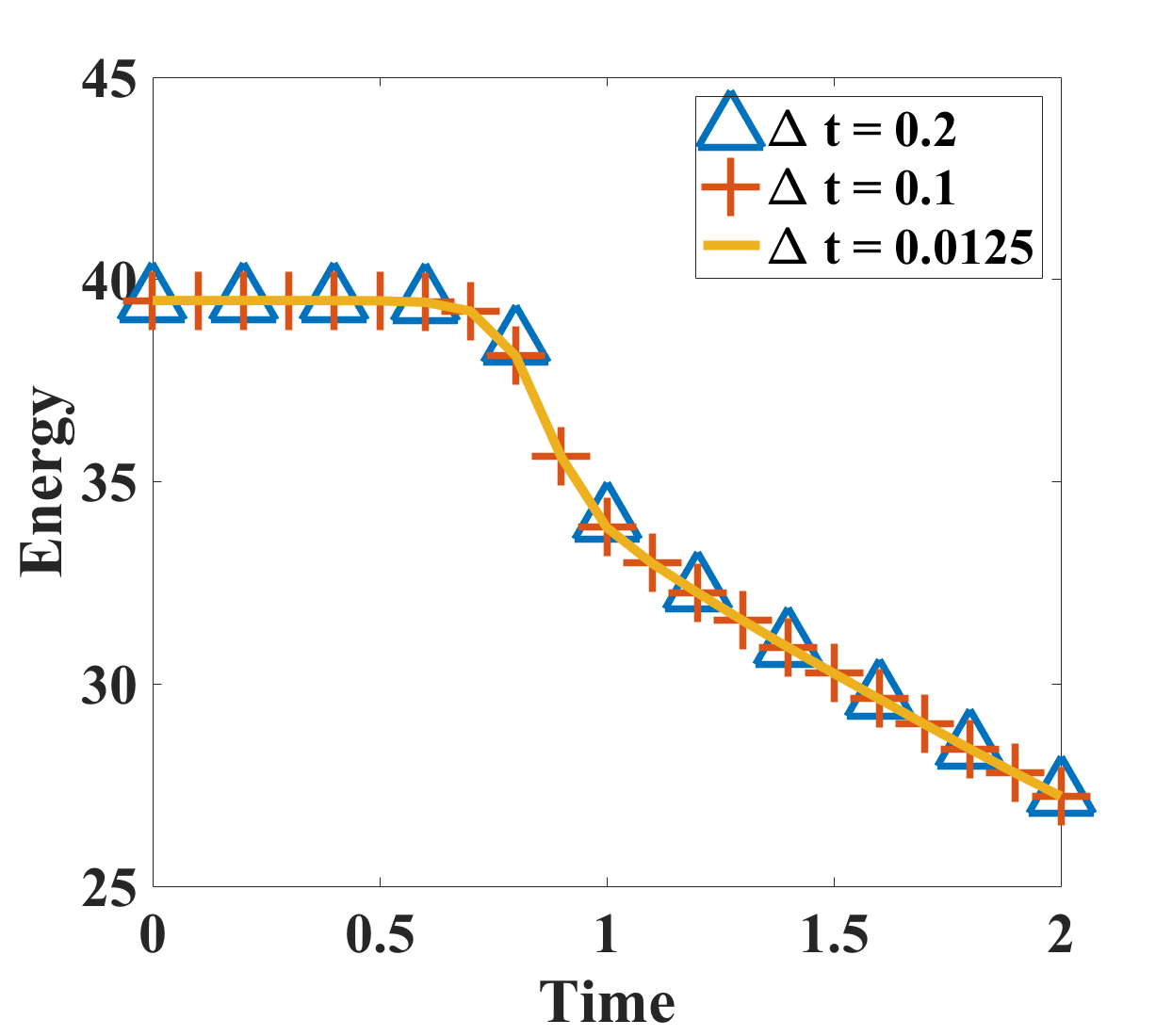}}
\subfigure[6th Order Gauss Method]{\includegraphics[height=4cm]{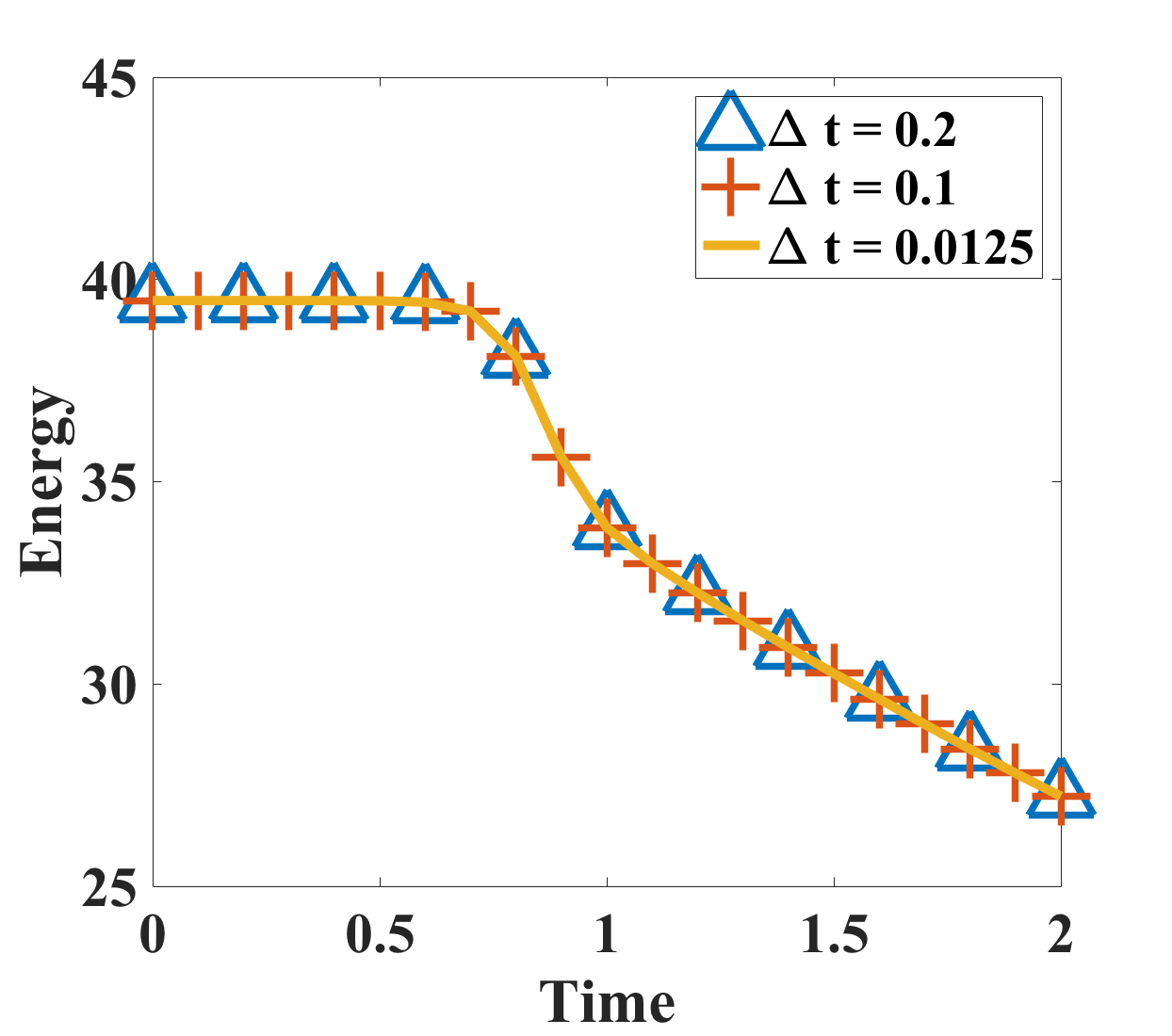}}

\caption{The energy calculated using various schemes in time period $[0, \,\,\, 2]$ . The figures show the energy predictions using (a) the LCN scheme; (b) ICN scheme; (c) fourth-order Gaussian collocation method; (d) sixth order Gaussian collocation method. Except for the LCN method, all other methods give a very good solution at large time step sizes.}
\label{fig:Coarsening-CN-Energy}
\end{figure}

One typical simulation using the 6th order Gaussian collocation method with a time step size $\Delta t = 0.01$ is depicted in Figure \ref{fig:Coarsening-Gauss-example}, where we observe fairly accurate prediction of the coarsening dynamics in various times.
\begin{figure}
\centering

\subfigure{
\includegraphics[height=4cm]{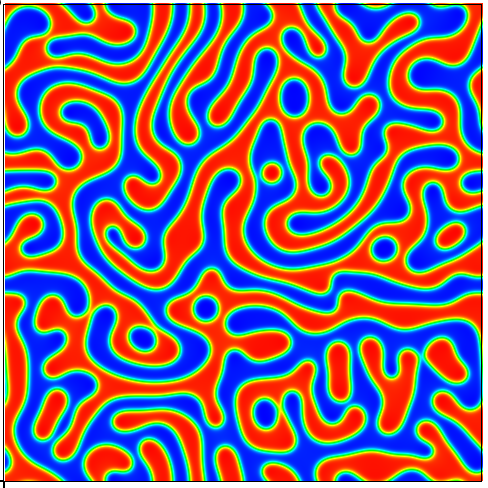}
\includegraphics[height=4cm]{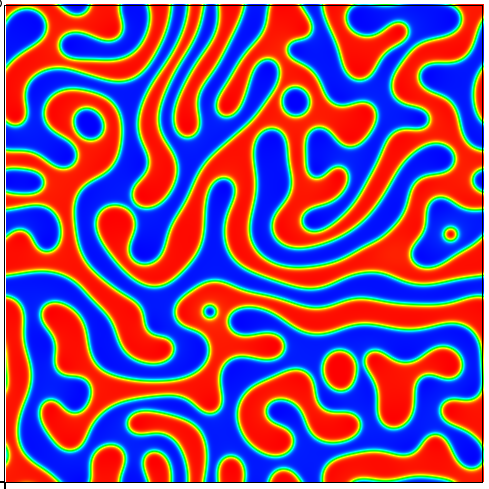}
\includegraphics[height=4cm]{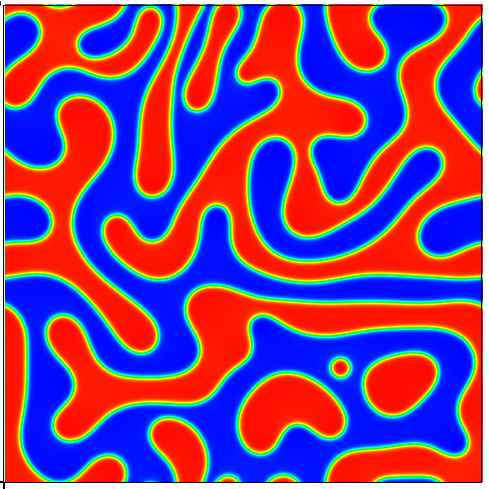}
}
\subfigure{
\includegraphics[height=4cm]{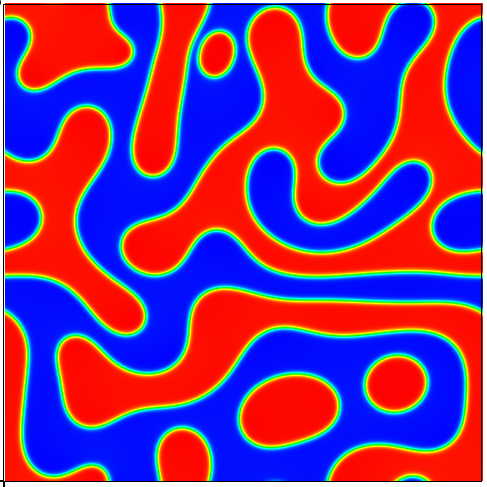}
\includegraphics[height=4cm]{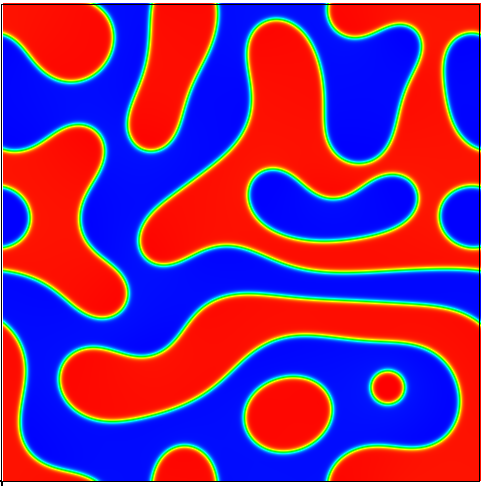}
\includegraphics[height=4cm]{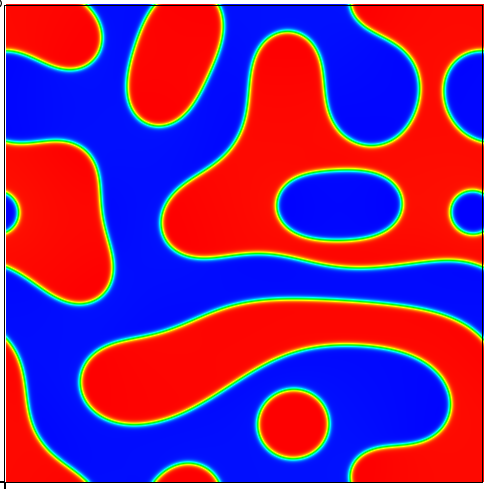}
}
\caption{The profile of $\phi$ during coarsening. Here red represents $\phi=1$ and blue represents $\phi=-1$. The profiles of $\phi$ at time $t=10,20,50,100,200,380$ are shown, respectively. }
\label{fig:Coarsening-Gauss-example}
\end{figure}

\textbf{Example 2: Allen-Cahn Equation.} Here we test the proposed numerical schemes on the  Allen-Cahn equation
\beq
\partial_t \phi = -\lambda \Big( -\varepsilon^2 \Delta \phi + \phi^3 - \phi \Big).
\eeq
Introducing the same auxiliary variable, $q =\frac{1}{\sqrt{2}} (\phi^2 -1 - \gamma_0),$ we obtain the reformulated EQ system
\beq
\left\{
\bea{l}
\partial_t \phi = -\lambda  ( -\varepsilon^2 \Delta \phi + \gamma_0 \phi + q g(\phi) ), \\
\partial_t q = g(\phi) \partial_t \phi, \quad g(\phi) = \sqrt{2} \phi,
\eea
\right.
\eeq
which can be written into
\beq
\frac{d}{d t} \Psi = \cA^T \cG \cA \cB \Psi, \quad \Psi=(\phi, q)^T, \quad \cA = (1 , g(\phi)), \quad \cB = \textrm{diag}(-\varepsilon^2 \Delta +\gamma_0, 1), \quad \cG= -\lambda.
\eeq

Here we test the numerical schemes via a benchmark example used in \cite{Che.S98}. We choose the parameter values as $\lambda=1, \varepsilon=1$, $L_x=L_y=256$. We use   $256\times 256$ mesh points to reduce the spatial error. The initial condition is chosen as a disk,
\beq
\phi(x,y,0) = \left\{
\bea{l}
1, \quad x^2 + y^2 < 100^2 \\
-1, \quad x^2 + y^2 \geq 100^2.
\eea
\right.
\eeq
It has been shown  that the radius $R$ of the disk at time $t$ is given as $R=\sqrt{R_0^2-2t}$ \cite{Che.S98}, where $R_0$ is the initial radius. In other words, the volume of the disk is given as $V = \pi R_0^2 - 2\pi t$. To test it, we implement and compare the proposed schemes with different time steps. The calculated results using the second-order CN schemes are summarized in Figure \ref{fig:AC-CN}(a-b), where we observe that when $\Delta t=0.0125$, the LCN scheme can predict the correct $-2\pi$ slope for the volume decreasing rate. For the ICN scheme, even with $\Delta t=0.5$, it predicts the correct volume decreasing rate. For the BDF4 and BDF6 schemes, they require approximately $\Delta t =0.025$ to predict the volume decreasing rate accurately, as shown in Figure \ref{fig:AC-CN}(c-d). The calculated volumes using  the HEQ schemes with different time step sizes are summarized in Figure \ref{fig:AC-CN}(e-f). It shows the Gaussian collocation method is more accurate such that even with time step size $\Delta t = 5$, it predicts a very accurate volume decaying rate $-2\pi$.

\begin{figure}
\centering

\subfigure[]{\includegraphics[height=5cm,width=5cm]{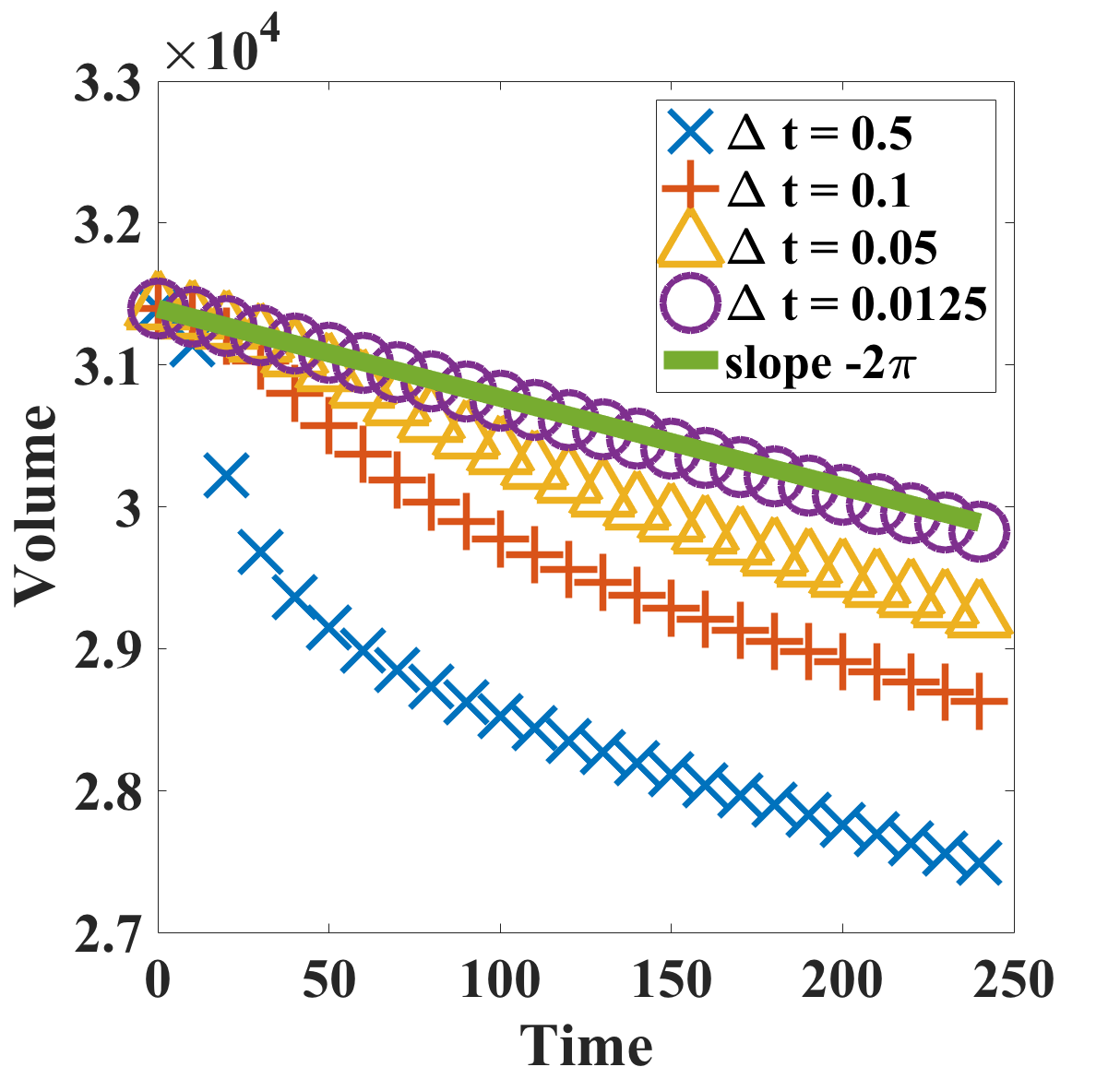}}
\subfigure[]{\includegraphics[height=5cm,width=5cm]{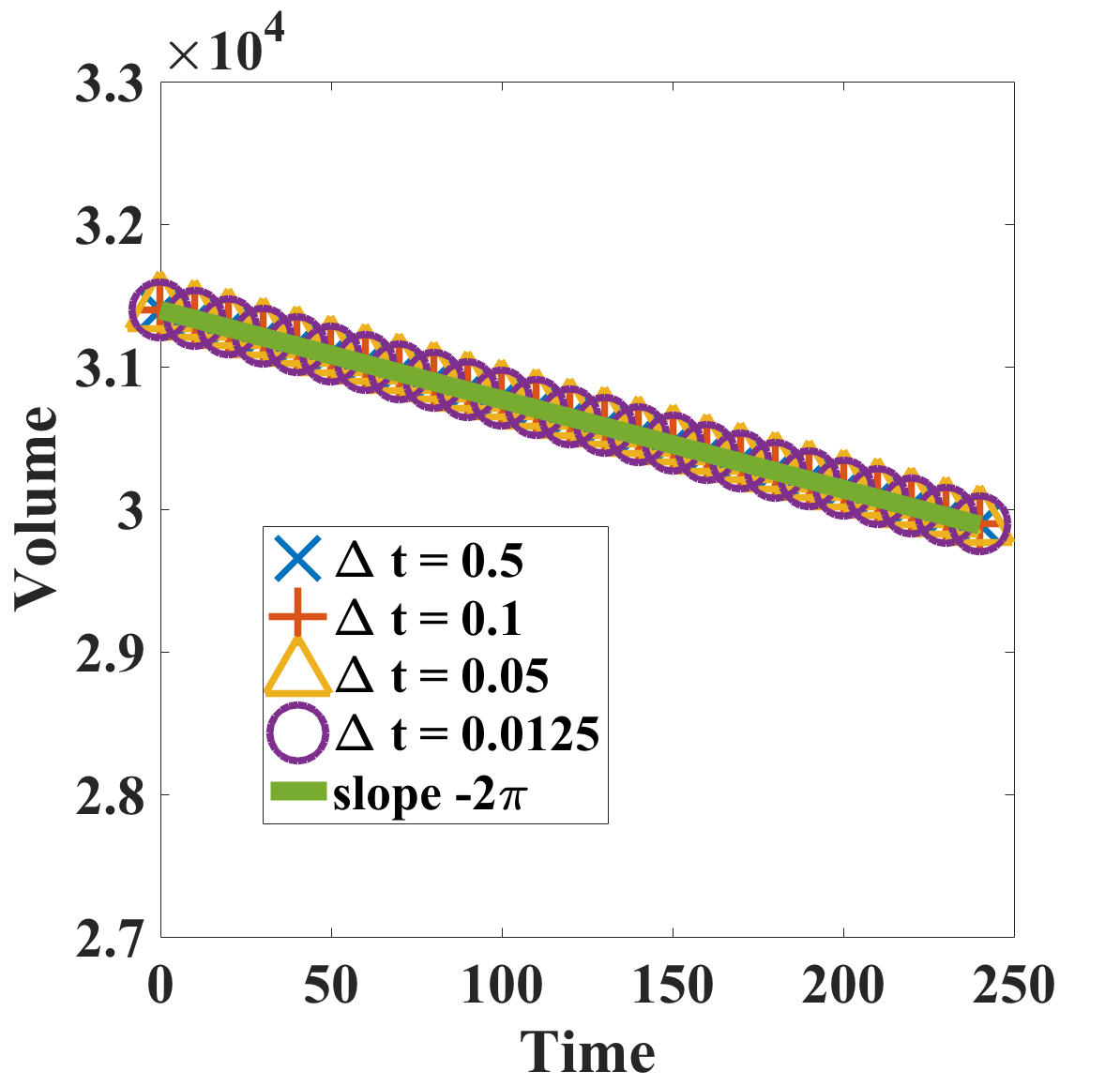}}
\subfigure[]{\includegraphics[height=5cm,width=5cm]{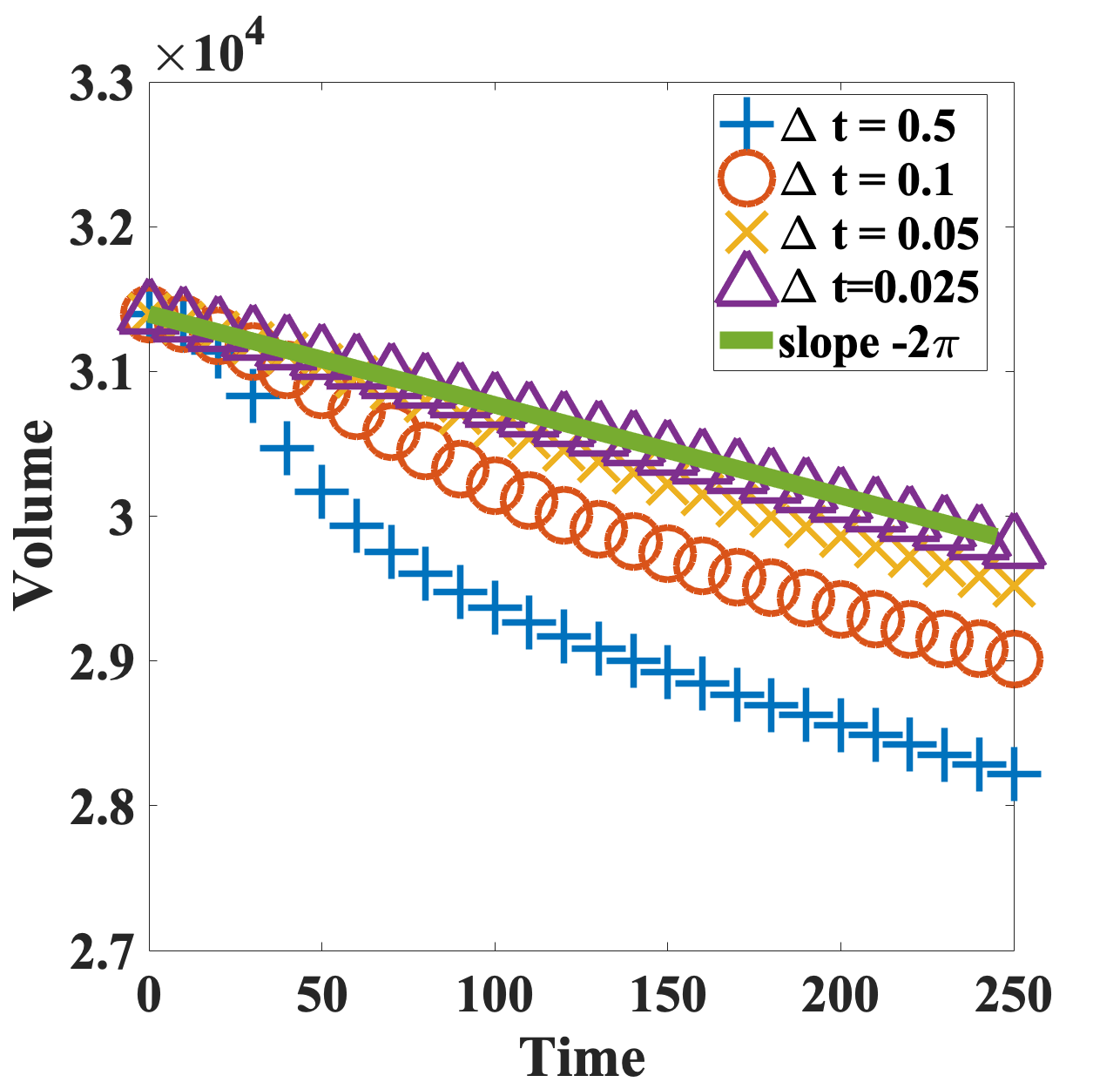}}

\subfigure[]{\includegraphics[height=5cm,width=5cm]{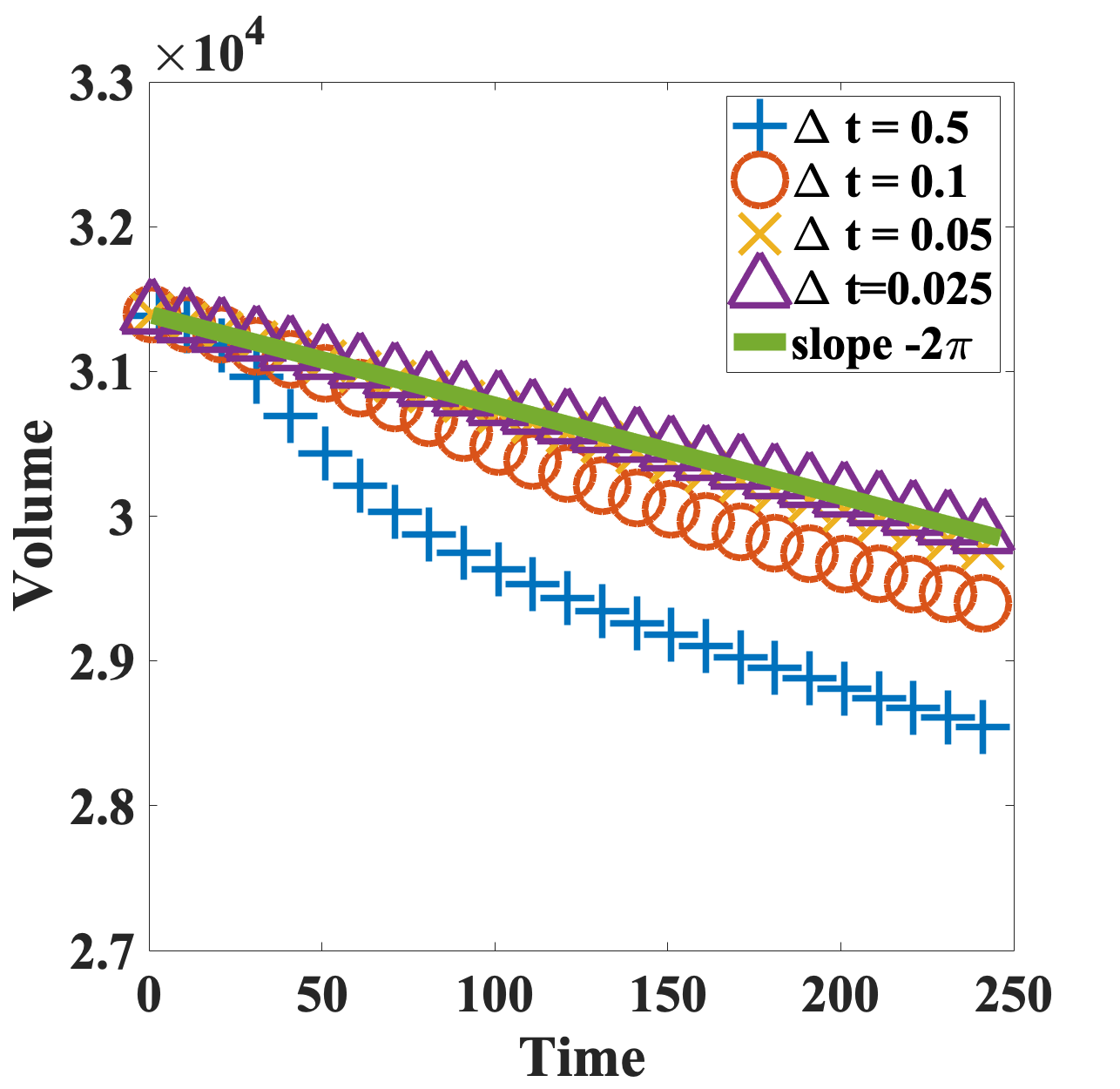}}
\subfigure[]{\includegraphics[height=5cm,width=5cm]{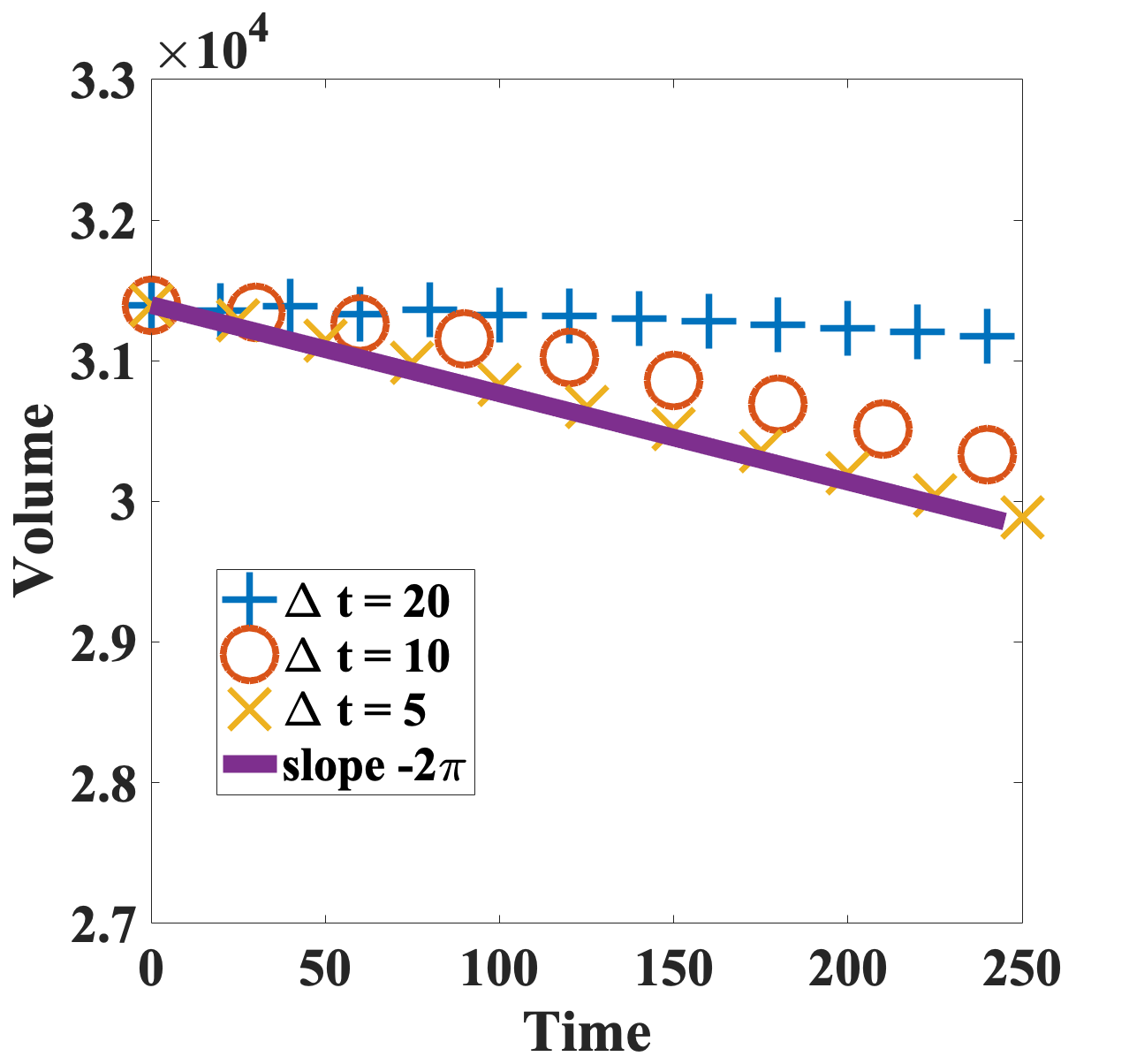}}
\subfigure[]{\includegraphics[height=5cm,width=5cm]{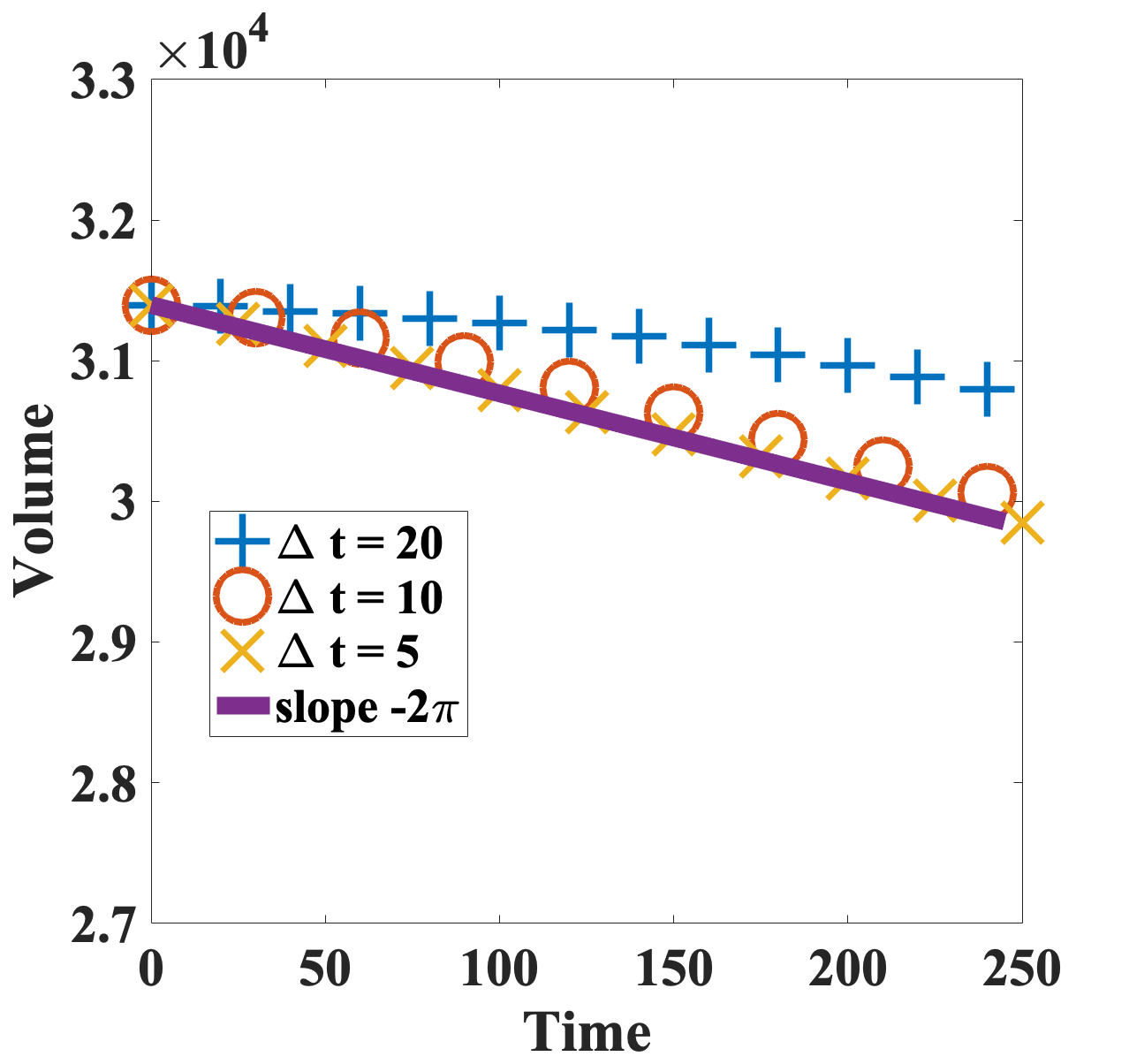}}

\caption{Allen-Cahn dynamics of a shrinking disk. The calculated volume using (a) the LCN scheme; (b) the ICN scheme; (c) the BDF4 scheme; (d) the BDF6 scheme; (e) the 4th order HEQ Collocation scheme; (f) the 6th order HEQ Collocation scheme.}
\label{fig:AC-CN}
\end{figure}

In particular, the temporal evolution of the disk using ICN scheme with time step size $\Delta t=0.1$ is plotted in Figure \ref{fig:AC-CN-phi}.
\begin{figure}
\centering

\subfigure{
\includegraphics[height=3cm]{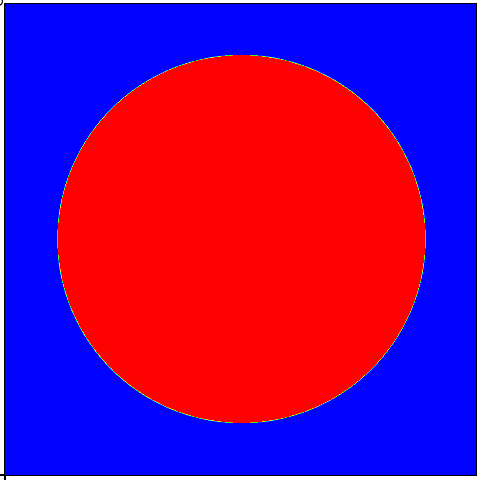}
\includegraphics[height=3cm]{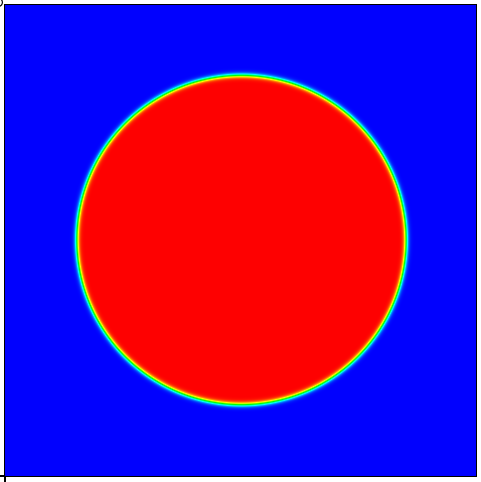}
\includegraphics[height=3cm]{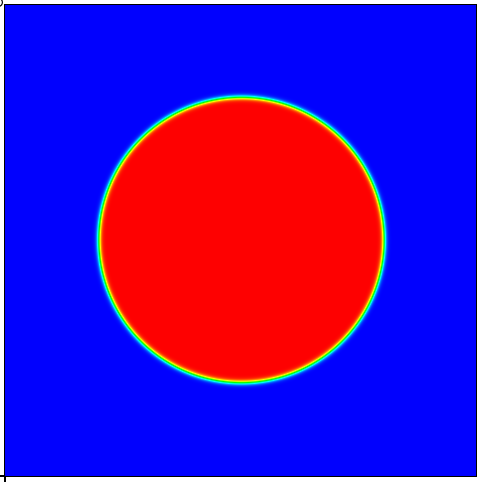}
\includegraphics[height=3cm]{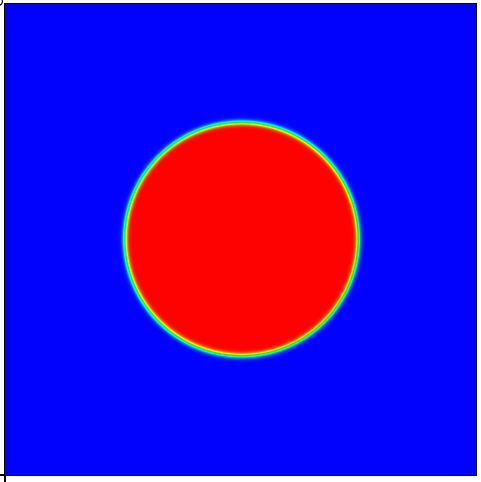}
}
\caption{Time evolution of a disk  driven by mean curvature. The profile of $\phi$ at time $t=0,1000,2000,3000$ are depicted.}
\label{fig:AC-CN-phi}
\end{figure}

\textbf{Example 3: Phase Field Crystal Growth Model.}
Next we solve the phase field  crystal growth model by the proposed schemes. The phase field crystal growth model was introduced in \cite{Elder2002,Elder2004,Vignal&Dalcin&Brown&Collier&CaloCAS2015} in the form of a Cahn-Hilliard equation \eqref{CH} with the free energy
\beq
F = \left(\frac{1}{4}\phi^4 - \frac{a}{2}\phi^2,1\right) + \frac{c}{2}\big(\|\phi\|^2 - 2\|\nabla\phi\|^2 + \|\Delta\phi\|^2\big).
\eeq
We set the parameter values $a=0.325, c=1.$ We use the initial conditions and parameter values given in \cite{Vignal&Dalcin&Brown&Collier&CaloCAS2015}, i.e.
$\phi_0 (\bx)= \overline{ \phi} + \omega(\bx) (A \phi_s(\bx))$,
$\phi_s(\bx) = \cos( \frac{k}{\sqrt{3}}y) \cos(k x) - \frac{1}{2} \cos (\frac{2k}{\sqrt{3}} y)$,
where $k$ represents a wavelength related to the lattice constant, $A$ represents an amplitude of the fluctuations in density, and the scaling function $\omega(\bx)$ is defined as
\beq
\omega(\bx) =
\left\{
\bea{rcl}
&& \Big(1 - (\frac{\| \bx - \bx_0\|}{d_0})^2 \Big)^2, \quad \mbox{if} \quad \| \bx - \bx_0\| \leq d_0, \\
&& 0, \quad \mbox{otherwise,}
\eea
\right.
\eeq
$d_0$ is a prescribed parameter.
In the simulation, we choose $\Omega =[0, \,\, 150]^2$, $\varepsilon = 0.325$, $\overline{ \phi} = \frac{\sqrt{\varepsilon}}{2}$, $A = \frac{4}{5} \Big( \overline{ \phi} + \frac{ \sqrt{15 \varepsilon - 36 \overline{ \phi}^2}}{3} \Big)$, $d_0=25$, $k=\frac{\sqrt{3}}{2}$ and $512\times 512$  meshes. A numerical simulation with the 4th order Gaussian collocation method is summarized in Figure \ref{fig:PFC}, where we observe qualitatively similar predictions as reported in   \cite{Wis09,PhaseFieldCrystal,Yang-EQ-PFC}.

\begin{figure}
\centering

\subfigure{
\includegraphics[height=4cm]{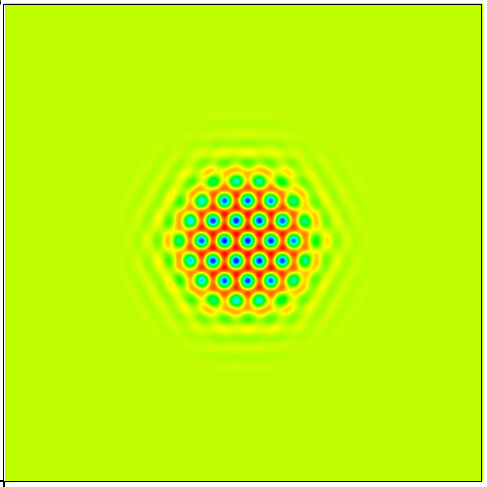}
\includegraphics[height=4cm]{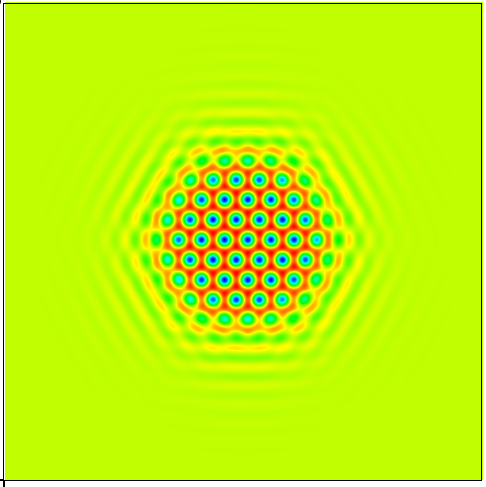}
\includegraphics[height=4cm]{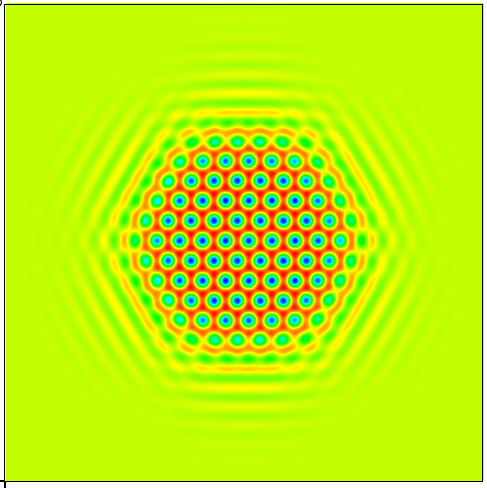}
}

\subfigure{
\includegraphics[height=4cm]{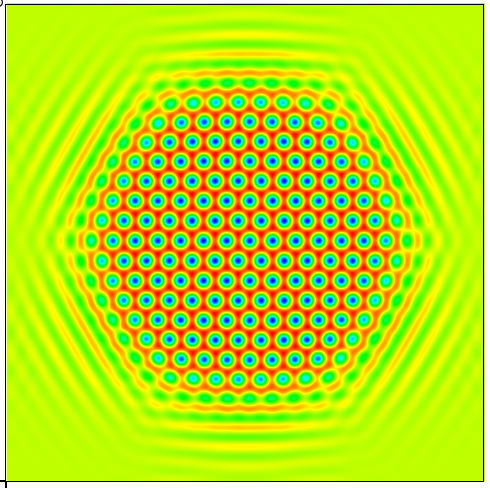}
\includegraphics[height=4cm]{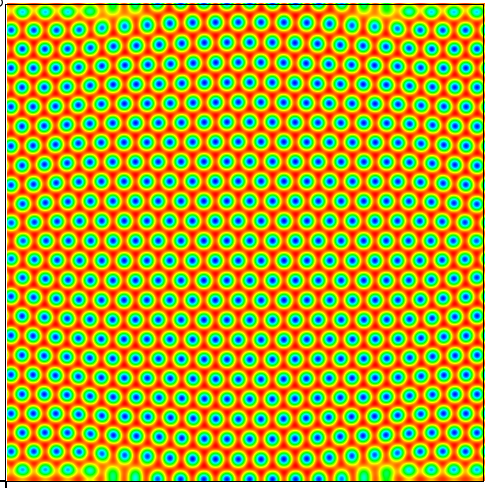}
\includegraphics[height=4cm]{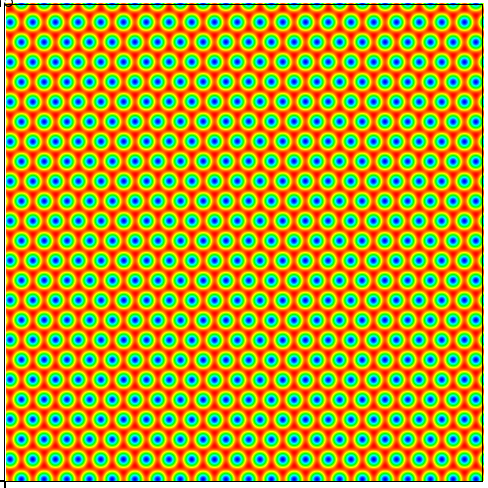}
}
\caption{The phase field crystal growth dynamics at times $t=10,20,30,50,100,1000$. }
\label{fig:PFC}
\end{figure}

\textbf{Example 4: Molecular Beam Epitaxy Model.}
Next, we study the molecular beam epitaxy (MBE) model \eqref{MBE}.
With  $\phi$ representing the scaled height function of the thin film, the continuum MBE model reads as 
\beq
\partial_t \phi = -M \Big( \varepsilon^2 \Delta \phi  + \nabla \cdot \Big( (1-|\nabla \phi|^2)\nabla \phi) \Big), (\bx,t) \in \Omega \times (0, T].
\eeq
By introducing $q =\frac{1}{\sqrt{2}} \Big( |\nabla \phi|^2 -1 - \gamma_0 \Big)$, the schemes proposed in previous sections could be readily applied to this model. Here we omit the details for simplicity.

To compare the schemes, we carry out a standard benchmark problem used in \cite{Li&Liu2003, Li&LiuJNS2004}. 
Given the domain $[0, 2\pi]^2$ and parameter values $\lambda=1$, $\epsilon^2=0.1$, and initial profile
$\phi(x,y) = 0.1 ( \sin(2x)\sin(2y) + \sin(5x)\sin(5y) )$.  We use $256 \times 256$ meshes. The comparisons are summarized in Figure \ref{fig:MBE}. It shows that even though the LCN scheme is unconditionally energy stable, it predicts erratic coarsening dynamics with large time steps. Also, the BDF-6 scheme requires relatively small time step to predict the accurate energy curve. On the contrary, the ICN scheme and the HEQ Gaussian collocation scheme can predict robust dynamics with larger time steps. As we have alluded to earlier that the ICN scheme only increases the CPU time slightly for up to 5 times of an FFT  solve in each predictor-corrector step \eqref{CN-P2}.

\begin{figure}
\centering
\subfigure[LCN scheme]{\includegraphics[height=3.75cm,width=3.75cm]{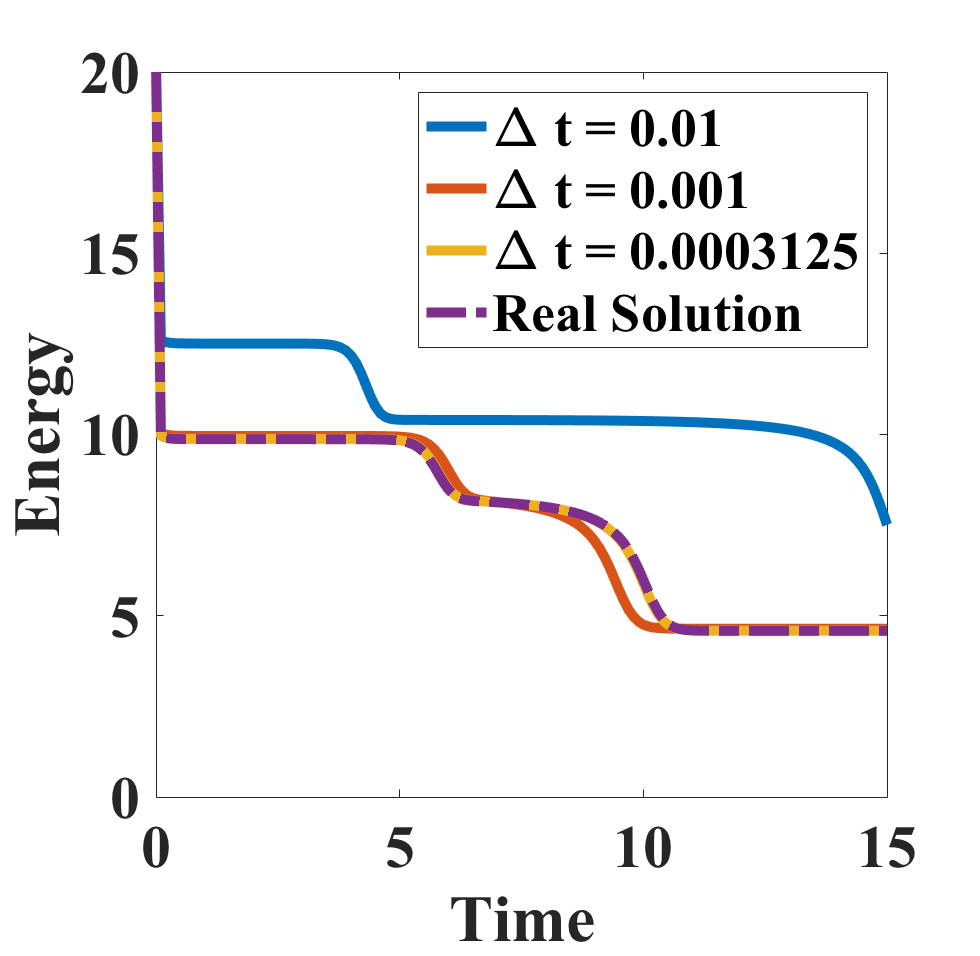}}
\subfigure[ICN scheme]{\includegraphics[height=3.75cm,width=3.75cm]{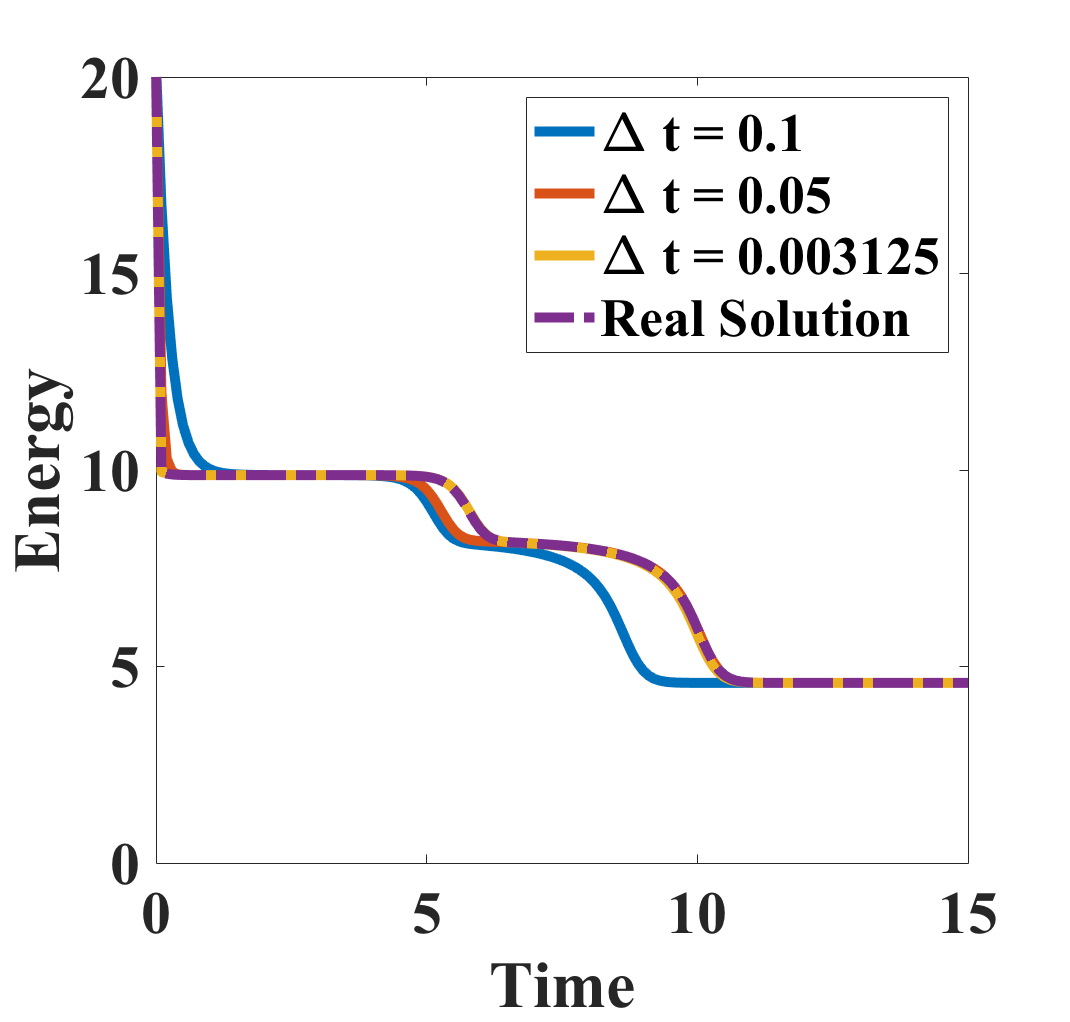}}
\subfigure[BDF6 scheme]{\includegraphics[height=3.75cm,width=3.75cm]{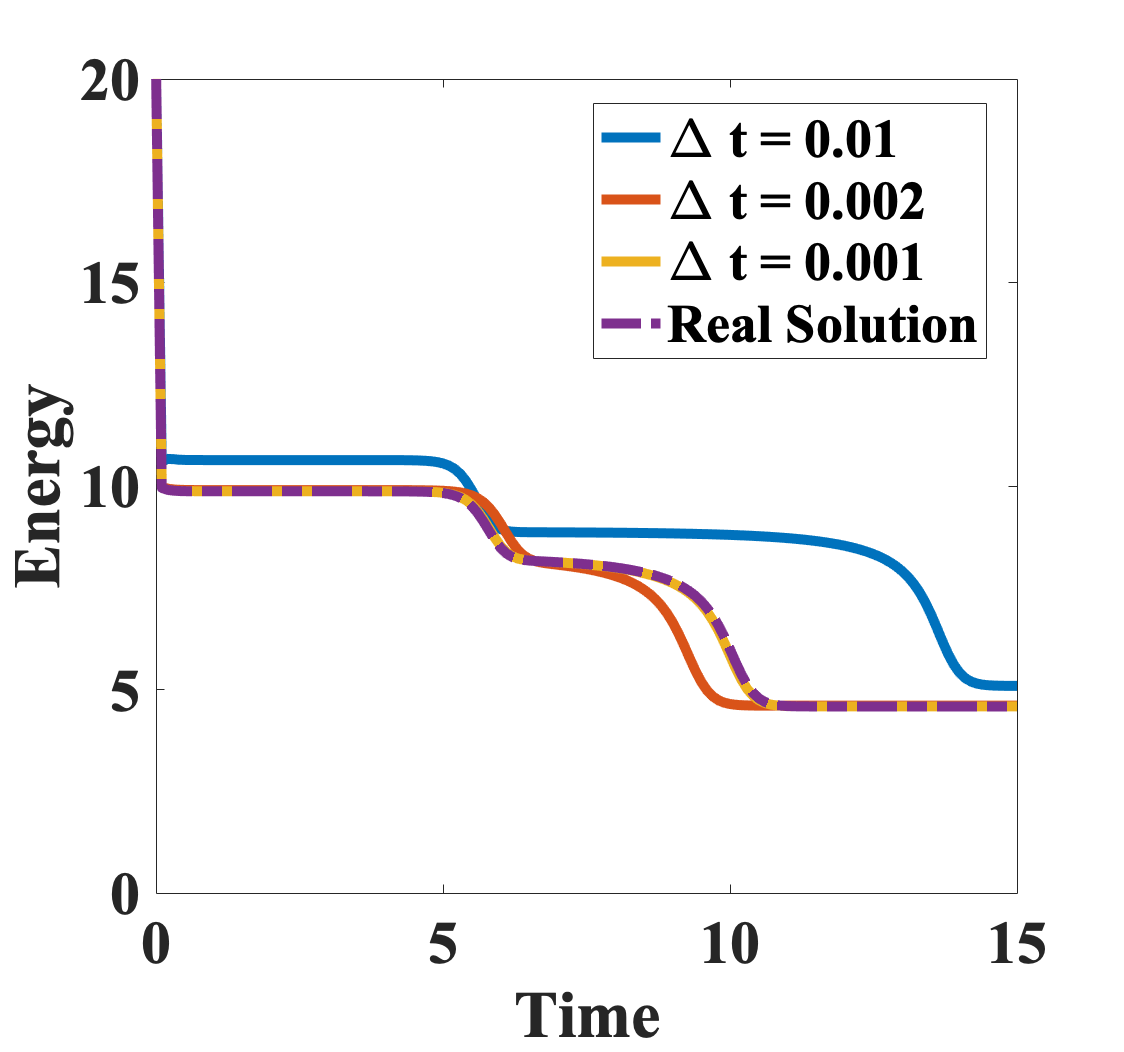}}
\subfigure[4th-order Gauss scheme]{\includegraphics[height=3.75cm,width=3.75cm]{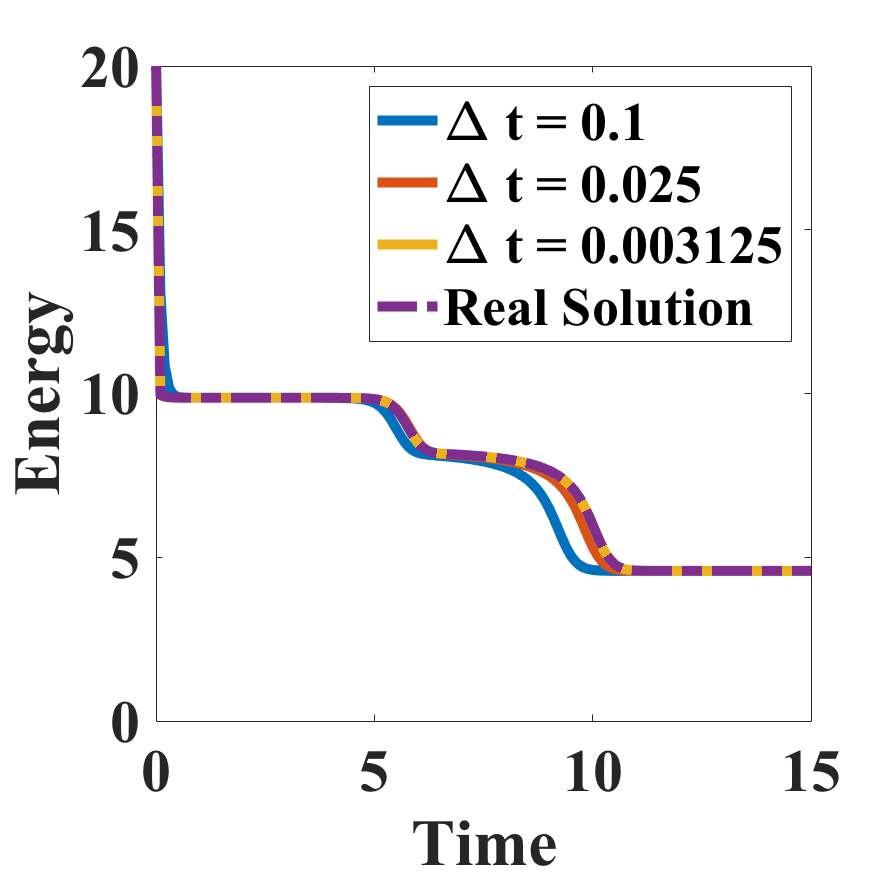}}
\caption{A comparison of free energies in MBE using different schemes with respect to various time steps. This figure shows both the ICN scheme and the HEQ Gaussian collocation method predict fairly  accurate energy profiles with relative large time steps, while the LCN scheme and the BDF6 scheme require finer time steps to reach the similar accuracy.}
\label{fig:MBE}
\end{figure}

%
%

\section{Conclusion}\label{sec:Con}

In this paper, we have demonstrated two general strategies to derive unconditionally energy-stable numerical approximations for thermodynamically consistent gradient flow models. In the first strategy, we present a prediction-correction approach to derive effective, linear schemes to solve the nonlinear thermodynamically consistent gradient flow models and in the meantime improve the accuracy of the schemes. The new schemes retain second-order accurate in time, and their errors are much smaller than those of the linear schemes. Moreover, the 2nd order convergence rate is attained even with a large time step.  In the second strategy, we propose a novel idea utilizing the quadratic-invariant preserving Runge-Kutta multi-stage discretization, resulting in arbitrary order, unconditionally energy stable numerical approximations for a general class of gradient flow models, which we named HEQ schemes. Unconditional energy stability is established rigorously.

Several numerical experiments for gradient flow problems are presented to illustrate the accuracy and efficiency of the numerical schemes. The gradient flow models include the Cahn-Hilliard equation with the Ginzburg-Landau double well free energy, the Allen-Cahn equation with the same energy functional, the crystal growth model, and the molecular beam epitaxy grwoth model. These models are solved with the proposed schemes, and numerical comparisons are presented. Through numerical experiments, it is clear that the prediction-correction CN scheme yields a significantly smaller error than the linear CN scheme, while the prediction-correction BDF2 scheme demonstrates the similar improvement over the lower order method at large time steps.
This newly proposed prediction-correction schemes and the high-order schemes based on Runge-Kutta methods are rather general that they can be readily applied to a large class of thermodynamically consistent models. Also, instead of using the EQ approach, the proposed numerical strategies here could easily be applied to the reformulated models using the SAV approach to derive high-order, efficient energy stable schemes.

\section*{Acknowledgment}
Yuezheng Gong's work is partially supported by the Natural Science Foundation of Jiangsu Province (Grant No. BK20180413) and the National Natural Science Foundation of China (Grant No. 11801269). Qi Wang's research is partially supported by NSF-DMS-1517347,  DMS-1815921, and OIA-1655740 award, NSFC awards \#11571032, \#91630207 and NSAF-U1530401. Jia Zhao's work is partially supported by National Science Foundation under grant number NSF DMS-1816783.


\end{document}